\documentclass[10pt,a4paper]{article}
\usepackage[T1]{fontenc}
\usepackage{lmodern}
\usepackage{textcomp}
\usepackage{amsmath}
\usepackage{amsthm}
\usepackage{textgreek}
\usepackage{amsfonts}
\usepackage{algorithm}
\usepackage{algpseudocode}
\newcommand{\hlabel}{\phantomsection\label}

\usepackage{tabularray}
\usepackage{float}
\usepackage{subcaption}
\usepackage{rotating}
\usepackage{soul}
\usepackage{hyperref}
\usepackage{tabularx}
\usepackage{multicol}
\usepackage{tikz}
\usepackage{amssymb}
\usepackage{blindtext}
\usepackage{graphicx}
\usepackage[shortlabels]{enumitem}
\usepackage[left=2.5cm,right=2.5cm,top=2cm,bottom=2cm]{geometry}
\usepackage{booktabs}
\usepackage{multirow}
\usepackage{rotating,tabularx}
\usepackage{longfbox}
\newtheorem{Assumption}{Assumption}

\newtheorem{Claim}{Claim}
\newtheorem{Theorem}{Theorem}[section]
\newtheorem{Proposition}[Theorem]{Proposition}
\newtheorem{Remark}[Theorem]{Remark}
\newtheorem{Lemma}[Theorem]{Lemma}

\newtheorem{Definition}[Theorem]{Definition}

\newtheorem{Algorithm}{Algorithm}

\usepackage{hyperref}
\hypersetup{colorlinks=true,urlcolor=red}
\expandafter\let\expandafter\oldproof\csname\string\proof\endcsname
\let\oldendproof\endproof
\renewenvironment{proof}[1][\proofname]{
\oldproof[\ttfamily\scshape \bf #1.]
}{\oldendproof}
\newcommand{\set}[1]{\left\{#1\right\}}

\def\rin{{\rm in}}
\def\rex{{\rm ex}}
\def\tilde{\widetilde}
\def\emp{\emptyset}

\def\dom{{\rm dom}\,}

\def\B{\mathbb B}

\def\ox{\overline{x}}

\def\disp{\displaystyle}

\def\Bar{\overline}

\def\epsilon{\varepsilon}
\def\ox{\bar{x}}

\def\C{\mathcal{C}}

\def\prox{\mbox{\rm Prox}}

\def\gph{\mbox{\rm gph}\,}

\def\dom{\mbox{\rm dom}\,}

\def\emp{\emptyset}
\def\st{\stackrel}

\def \N{{\rm I\!N}}
\def \R{{\rm I\!R}}

\newcommand{\dotproduct}[1]{\left\langle#1\right\rangle}
\newcommand{\brac}[1]{\left(#1\right)}
\newcommand{\sbrac}[1]{\left[#1\right]}
\newcommand{\abs}[1]{\left|#1\right|}
\newcommand{\norm}[1]{\left\|#1\right\|}
\numberwithin{equation}{section}

\title{Convergence of First-Order Algorithms with Momentum from the Perspective of an Inexact Gradient Descent Method}
\author{Pham Duy Khanh\footnote{Department of Mathematics, Ho Chi Minh City University of Education, Ho Chi Minh City, Vietnam. E-mail:  khanhpd@hcmue.edu.vn}\quad Boris S. Mordukhovich\footnote{Department of Mathematics, Wayne State University, Detroit, Michigan, USA. E-mail: aa1086@wayne.edu. Research of this author was partly supported by the US National Science Foundation under grant DMS-2204519 and by the Australian Research Council under grant DP-250101112.}\quad Dat Ba Tran\footnote{Department of Mathematics, Wayne State University, Detroit, Michigan, USA. E-mail: tranbadat@wayne.edu. Research of this author was partly supported by the US National Science Foundation under grant DMS-2204519.}}
\begin{document}
\maketitle
\vspace*{-0.3in}
\noindent
{\small{\bf Abstract}. This paper introduces a novel inexact gradient descent method with momentum (IGDm) considered as a general framework for various first-order methods with momentum. This includes, in particular, the inexact proximal point method (IPPm), extragradient method (EGm), and sharpness-aware minimization (SAMm). Asymptotic convergence properies of IGDm are established under both global and local assumptions on objective functions with providing constructive convergence rates depending on the Polyak-\L ojasiewicz-Kurdyka (PLK) conditions for the objective function. Global convergence of EGm and SAMm for general smooth functions and of IPPM for weakly convex functions is derived in this way. Moreover, local convergence properties of EGm and SAMm for locally smooth functions as well as of IPPm for prox-regular functions are established. Numerical experiments for derivative-free optimization problems are conducted to confirm the efficiency of the momentum effects of the developed methods under inexactness of gradient computations.}\\[1ex] 
\vspace*{-0.1in}

\section{Introduction}\label{intro}\vspace*{-0.05in}

This paper addresses the design and convergence analysis of first-order optimization algorithms with momentum for solving unconstrained optimization problems of the type
\begin{align}\label{main problem}
{\rm minimize}\quad f(x)\quad \text{ subject to }x\in\R^n,
\end{align}
where $f:\R^n\rightarrow\overline{\R}$ is a continuously differentiable function on some open and convex subset $X$ of $\R^n$ having a Lipschitz continuous gradient with constant $L>0$ on $X$, i.e.,
\begin{align*}
\norm{\nabla f(x)-\nabla f(y)}\le L\norm{x-y}\;\text{ for all }\;x,y\in X.
\end{align*}
First-order algorithms offer optimization techniques that leverage information of function values and gradients/subgradients while excluding Hessian information of the function. The most widely used first-order optimization methods/algorithms are the gradient descent (GD) method \cite{Cauchy1847,Polyak1987}, the inexact proximal point method (IPPM) \cite{Rockafellar1976} and its variants including the proximal gradient method (PGM) \cite{BeckTeboulle2009, Nesterov2013}, and the augmented Lagrangian method (ALM) \cite{Rockafellar1976a}. An increasing number of real-world applications can be modeled as large-scale optimization problems,  and thus first-order algorithms have become dominant due to their simplicity, flexibility, and moderate efficiency. Since the introduction of conjugate gradient methods, it has been realized that the gradient descent method can be improved by incorporating information of past iterations. Modern research on such extended first-order methods dates to Polyak \cite{Polyak1964} whose heavy ball method incorporates a momentum term into the gradient step. Over the years, the methods of this type, which are usually labeled as gradient descent methods with momentum, have  been developed intensively. The breakthrough achievements in this vein are Nesterov's acceleration methods \cite{Nesterov1983,Nesterov2018}. The momentum effect is incorporated more and more in optimization methods other than of the gradient descent type, e.g., in  the proximal point \cite{Guller1992} and proximal gradient methods \cite{BeckTeboulle2009,Nesterov2013}.

In this paper, we introduce a novel {\em inexact gradient descent method with momentum} (IGDm) for optimization of nonconvex functions. The momentum effect considered in our analysis is general and allows us to encompass Polyak's heavy ball methods \cite{Polyak1964}, different versions of Nesterov's acceleration \cite{Nesterov2018}, FISTA \cite{BeckTeboulle2009}, and the recent inertial algorithm introduced by L\'asl\'o \cite{Laslo}. The \textit{relative inexact condition} used in IGDm is not only standard and simple but also allows us to encompass fundamental and efficient first-order optimization methods including the inexact proximal point method \cite{Rockafellar1976}, Korpelevich's
extragradient method \cite{Korpelevich1976}, and a recently developed efficient optimization technique known as sharpness-aware minimization (SAM) \cite{AndriuschenkoFlammarion2022,ForetKleinerMobahiNeyshabur2021}. The relative inexact condition eventually has a broader range of applications than what will be theoretically demonstrated in this paper due to its appearance in the context of gradient-based methods for smooth functions, especially in derivative-free optimization scenarios. In such context, the exact gradient is often unavailable while leading us to the utilization of inexact gradients through approximation methods such as finite differences, linear interpolation, and Gupal estimation, which all are special cases of the inexact condition introduced in \cite{KhanhMordukhovichTran}. As shown in the recent publications \cite{KhanhMordukhovichPhatTran,KhanhMordukhovichTrana,KhanhMordukhovichTran2023}, the relative inexact condition is also efficient in the augmented Lagrangian method, inexact proximal gradient method, etc.

The convergence analysis for the IGDm method and its specifications developed below consists of two parts. The first part addresses the global version of IGDm for minimizing smooth functions with gradients being Lipschitz continuous globally on $\R^n$, i.e., of class $\mathcal{C}^{1,1}_L$. Convergence results obtained for this class include stationarity of accumulation points in a general setting as well as global convergence with constructive convergence rates under appropriate PLK conditions for objective functions. The second part of our analysis focuses on local convergence of IGDm when objective functions are smooth and their gradients are locally Lipschitzian around local minimizers, i.e., of class $\mathcal{C}^{1,1}$. For such functions, we establish local convergence with constructive convergence rates for IGDm under PLK conditions at local minimizers. Observe that the imposed PLK conditions constitute rather mild regularity properties that often hold for major classes of optimization problems and applications; see Section~\ref{sec prelim} for more discussions.

To demonstrate the universality of our analysis, we also introduce and simultaneously investigate convergence properties of the extragradient method with momentum (EGm), sharpness-aware minimization with momentum (SAMm), and inexact proximal point method with momentum (IPPm). The table below indicates the classes of functions and the type of convergence achieved for these methods.

\begin{table}[H]
\centering
\begin{tblr}{
column{3} = {c},
column{4} = {c},
cell{2}{1} = {r=2}{},
cell{4}{1} = {r=2}{},
cell{6}{1} = {r=2}{},
cell{8}{1} = {r=2}{},
}\hline
     & Version & Class of functions            & Result  \\\hline
IGDm & Global  & $\C^{1,1}_L$                           & Theorem~\ref{global IGDm}\\\hline
     & Local   & $\C^{1,1}$ + PLK                      & Theorem~\ref{local con IGDm}\\\hline
EGm  & Global  & $\C^{1,1}_L$                            & Theorem~\ref{glo EG}\\\hline
     & Local   & $\C^{1,1}$ + PLK                      & Theorem~\ref{lo EG} \\\hline
SAMm & Global  & $\C^{1,1}_L$                            & Theorem~\ref{glo SAM} \\\hline
     & Local   & $\C^{1,1}$ + PLK                      & Theorem~\ref{lo SAM} \\\hline
IPPm & Global   & Weakly convex                 & Theorem~\ref{theo IPPM global} \\\hline
     & Local  & Prox-regular + semi-algebraic & Theorem~\ref{lo IPPm}\\\hline
\end{tblr}
\caption{Convergence results for first-order methods with momentum}
\end{table}\vspace*{-0.1in}
Let us now discuss the novelty and the importance of the obtained convergence results in comparison with other related studies. It should be emphasized first that, to the best of our knowledge, all the \textit{local convergence} results presented here are {\em entirely new}. Concerning proximal point methods, local convergence analysis for prox-regular functions remains an open question; the largest class of functions considered recently in \cite{Rockafellar2021,Rockafellar2023} consists of variationally convex ones, which form a strict subclass of prox-regular functions. The result in \cite[Theorem~4.2]{AttouchBolteSvaiter2013} addresses the convergence of an inexact proximal point method under a certain continuity assumption. Due to that assumption, this result fails when considering $\ell_0$-norm functions, which are typical examples of variational convex functions \cite{KhanhMordukhovichPhat2023} and prox-regular functions. Subsection~\ref{subsection IPPM} below provides a local study of the inexact proximal point method for prox-regular functions while examining not only its standard version but also variants when  it is incorporated with different types of momentum. Similarly, Subsections~\ref{subsection EGm} and \ref{SAMm} establish fundamental local convergence properties of the classical extragradient method \cite{Korpelevich1976} and the efficiently developed algorithms of 
sharpness-aware minimization \cite{ForetKleinerMobahiNeyshabur2021}. Such methods are especially important for applications to training processes of deep neural networks, AI, and other practical models; see, e.g., \cite{ForetKleinerMobahiNeyshabur2021,LinKongStichJaggi2020}.

Regarding {\em global convergence} results, the general framework of IGDm allows us to extend convergence behavior of the extragradient method for unconstrained nonconvex problems to its accelerated versions by employing a simple proof in Theorem~\ref{glo EG}. These new results are more advanced in comparison with those in \cite[Section~3]{NguyenPauwelsRichardSutter2018} in the sense that we now cover accelerated versions while the latter results only consider the basic version. Our results also remove the boundedness assumptions in \cite[Theorem~3.1]{NguyenPauwelsRichardSutter2018}. It should be emphasized that although the original extragradient method in \cite{Korpelevich1976} was designed for variational inequality problems, it has recently attracted attention to solve smooth unconstrained problems of machine learning \cite{LinKongStichJaggi2020}. Our global results on inexact proximal point methods also extend the recent ones in \cite{Kim2021} from convex to weakly convex functions. The newly obtained results provide convergence rates for the iterative sequence in Theorem~\ref{theo IPPM global}, which were not established in \cite{Kim2021}.

Yet another contribution of this paper is demonstrating numerical advantages of the proposed algorithms. To illustrate this, we examine the performance of IGDm methods to solve derivative-free optimization problems in both convex and nonconvex settings. Our findings indicate a general superiority of newly developed methods over the basic version of IGD (inexact gradient descent) method without momentum. As discussed in \cite{KhanhMordukhovichTrana,ShiXuanOztoprakNocedal2023}, IGD in general outperforms other well-developed methods in derivative-free optimization including FMINSEARCH, i.e., the Nelder–Mead simplex-based method from \cite{LagariasReedsWrightWright1998}, the implicit filtering algorithms \cite{GilmoreKelley1995}, and the random gradient-free algorithm for smooth optimization proposed by Nesterov and Spokoiny \cite{NesterovSpokoiny2017}. As a consequence, IGDm can be recommended as a preferable optimizer for derivative-free smooth (convex and nonconvex) optimization problems.\vspace*{0.03in}

The rest of the paper is organized as follows. Section~\ref{sec prelim} presents some preliminaries and auxiliary results used throughout the entire paper. In Section~\ref{sec: IGDm}, we introduce the inexact gradient descent method with momentum and provide its comprehensive global and local convergence analysis. Section~\ref{sec: illustrate} derives global and local convergence properties for a broad spectrum of first-order methods with momentum including the extragradient method, sharpness-aware minimization, and proximal point method, which are based on the analysis of IGDm. Numerical experiments demonstrating the efficiency of the momentum effect in IGDm for derivative-free optimization problems are presented and analyzed in Section~\ref{sec: numerical}. In Section~\ref{conc}, we summarize the main achievements of the paper and discuss some directions of our future research.\vspace*{-0.15in}

\section{Preliminaries and Auxiliary Results}\label{sec prelim}\vspace*{-0.05in}

First we recall some notions and notations frequently used in the paper. All our considerations are given in the space $\R^n$ with the Euclidean norm $\|\cdot\|$. We use the matrix norm given by
\begin{equation*}
\norm{A}:=\max\big\{\norm{Ax}\;\big|\;\norm{x}=1\big\}\;\text{ for any }\;m\times n \text{ matrix }A.
\end{equation*}
Along with the Euclidean norm $\norm{\cdot}$, we also utilize the $\ell_1$-norm $\norm{\cdot}_1$ defined by
\begin{equation*}
\disp\norm{x}_1=\sum_{k=1}^n|x_k|\;\text{ for all }\;x=(x_1,\ldots,x_n)\in\R^n.
\end{equation*}
As always, $\N:=\{1,2,\ldots\}$ signifies the collections of natural numbers. Given $r>0$, the open ball centered at $\bar x$ with radius $r$ is denoted by $\mathbb{B}(\bar x,\rho)$ and its closure is $\overline{\mathbb{B}}(\bar x,\rho)$.
The symbol $x^k\st{J}{\to}\ox$ means that $x^k\to\ox$ as $k\to\infty$ with $k\in J\subset\N$. Recall that $\bar x$ is a \textit{stationary point} and that $f(\bar x)$ is a \textit{critical value} of a $\mathcal{C}^1$-smooth function $f\colon\R^n\rightarrow\R$ if $\nabla f(\bar x)=0$. A $\C^1-$smooth function $f:\Omega\rightarrow {\R}$ is said to have a Lipschitz continuous gradient with the uniform constant $L>0$ on some subset $\Omega$ of $\R^n$ if 
\begin{align}\label{Lips def}
\norm{\nabla f(x)-\nabla f(y)}\le L\norm{x-y}\;\text{ for all }\;x,y\in\Omega.
\end{align}
A function $f\colon\R^n\rightarrow\R$ is {\em locally Lipschitz continuous} around $\bar x$  if there is a neighborhood $\Omega$ of $\bar x$ such that \eqref{Lips def} holds; $f$ {\em globally Lipschitz continuous} if \eqref{Lips def} holds for $\Omega=\R^n$. In what follows, we denote by $\mathcal{C}^{1,1}$ the class of smooth functions that have a {\em locally Lipschitz continuous gradient} around $\bar x$, and by $\mathcal{C}^{1,1}_L$ the class of smooth functions that have a {\em globally Lipschitz continuous gradient with constant} $L>0$.

Next we recall the subdifferential constructions for proper extended-real-valued functions $h:\R^n\rightarrow\Bar\R:=(-\infty,\infty]$ with $\dom h:=\{x\in\R^n\;|\;h(x)<\infty\}\ne\emp$ taken from \cite{Mordukhovich2006,Mordukhovich2018,RockafellarWets} and used below. The (Fr\'echet) \textit{regular subdifferential} of $ h $ at $\ox \in\dom  h $ is 
\begin{align}\label{frechet}
\widehat{\partial} h (\bar x):=\Big\{v\in \R^n\;\Big|\;\liminf_{x\rightarrow\bar x}\dfrac{ h (x)- h (\bar x)-\dotproduct{v,x-\bar x}}{\norm{x-\bar x}}\ge 0\Big\},
\end{align}
and the (Mordukhovich) \textit{limiting subdifferential} of $ h $ at $\bar x\in\dom  h $ is defined by 
\begin{align}\label{limiting subdiff}
\partial  h (\bar x):=\set{v\in\R^n\;\big|\;\exists\,x^k\rightarrow \ox,\  h (x^k)\rightarrow h (\bar x),\ \widehat{\partial}h(x^k)\ni v^k\rightarrow v}.
\end{align}
When $ h $ is $\mathcal{C}^1$-smooth around $\bar x$, both subdifferentials in \eqref{frechet} and \eqref{limiting subdiff} reduce to the gradient $\nabla  h (\bar x)$, while for convex functions they agree with the classical subdifferential of convex analysis. A necessary condition for $\bar x\in\R^n$ to be a local minimizer of $ h $ is 
$0\in \partial h (\bar x)$. Any point $\bar x$ satisfying the latter condition is known in the literature as an {\em M$($ordukhovich$)$-stationary} (or {\em limiting-stationary}) point of $h$. Since this paper does not deal with other stationary concepts for nonsmooth functions, we simply use the {\em stationary} point term in what follows.

The convergence analysis of first-order methods developed in the subsequent sections largely employs the following important notions and results. To begin with, we recall a useful property of functions with Lipschitzian gradients taken from \cite[Lemma~A.11]{IzmailovSolodov2014}.

\begin{Lemma}\label{lemma descent}
Given $f:\R^n\rightarrow\R$ and $x,y\in\Omega$, assume that $f$ is differentiable on the line segment $[x,y]$ and its derivative is Lipschitz continuous on this segment with constant $L>0$. Then we have
\begin{align}\label{eq descent}
\abs{f(y)-f(x)-\dotproduct{\nabla f(x),y-x}}\le \frac{L}{2}\norm{y-x}^2.
\end{align}
\end{Lemma}

The next notion plays a fundamental role in convergence analysis of algorithms. The reader is referred to \cite{bento25} for the history of such conditions and to \cite{AbsilMahonyAndrews2005,AttouchBolteRedontSoubeyran2010,karimi,LiPong2018} for alternative names and further details.

\begin{Definition}\rm \label{KL ine}\rm We say that a lower semicontinuous (l.s.c.) function $f:\R^n\rightarrow\overline{\R}$ satisfies the $($basic$)$ \textit{Polyak-\L ojasiewicz-Kurdyka} $(PLK)$ condition at $\bar x\in\dom \partial f$ if there exist $\eta\in (0,\infty]$, a neighborhood $U$ of $\bar x$, and a desingularizing concave function $\varphi:[0,\eta)\rightarrow[0,\infty)$ such that:

{\bf(i)} $\varphi(0)=0$.

{\bf(ii)} $\varphi$ is $\mathcal{C}^1$-smooth on $(0,\eta)$.

{\bf(iii)} $\varphi'>0$ on $(0,\eta)$.

{\bf(iv)} For all $x\in U$ with $f(\bar{x})<f(x)<f(\bar{x})+\eta$, we have 
\begin{align}\label{KL 2}
\varphi'\big(f(x)-f(\bar{x})\big)d(0,\partial f(x))\ge 1.
\end{align}
\end{Definition}

\begin{Remark}\rm\label{algebraic}If $f$ satisfies the PLK property at $\bar x$ with a neighborhood $U$, it is clear that the same property holds for any $\tilde x\in U\cap {\rm dom\; }\partial f$ satisfying $f(\tilde x)=f(\bar x).$
It has been realized that the PLK condition is satisfied in rather  broad settings. In particular, it holds at every 
nonstationary point of $f$, for semianalytic and more general functions known as definable in o-minimal structures; see, e.g., \cite{AbsilMahonyAndrews2005,AttouchBolteRedontSoubeyran2010,AttouchBolteSvaiter2013} and the references therein. As demonstrated in \cite[Section~2]{KhanhMordukhovichTran2023}, the PLK condition formulated in \cite{AttouchBolteRedontSoubeyran2010} is stronger than the one in Definition~\ref{KL ine}.
\end{Remark}

Note that condition (ii) ensures that, for any fixed $x\in(0,\eta)$ and $a\in(0,x]$, the function $\varphi'$ is integrable on $[a,x]$ and thus $\varphi(x)=\varphi(a)+\int_a^{x}\varphi'(t)dt$. Taking the limit as $a\rightarrow0^+$ and using  
$\varphi(0)=0$ give us $\varphi(x)=\int_0^{x}\varphi'(t)dt$. The following result plays an important role in our local convergence analysis.

\begin{Lemma}\label{isolation lemma}
Let $f:\R^n\rightarrow\overline{\R}$ be continuous around some local minimizer $\bar x$ and satisfy the PLK condition at $\bar x$. Then there is a neighborhood $U$ of $\bar x$ such that $f(\bar x)$ is the unique critical value of $f$ in $U$.
\end{Lemma}
\begin{proof}
The PLK condition for $f$ at $\bar x$ gives us $\eta\in(0,\infty],$ a neighborhood $U$ of $\bar x$, and a function $\varphi:[0,\eta)\rightarrow[0,\infty)$ satisfying all the conditions in Definition~\ref{KL ine}. Since $f$ is continuous around the local minimizer $\bar x,$ we assume, by shrinking $U$ if necessary, that $f(\bar x)\le f(x)<f(\bar x)+\eta.$ Then condition (iv) in Definition~\ref{KL 2} can be rewritten as
\begin{align*}
\varphi'(f(x)-f(\bar x))d(0,\partial f(x))\ge 1\text{ whenever }x\in U,\;f(x)\ne f(\bar x).
\end{align*}
Therefore, $0\notin \partial f(x)$ for all such $x$, i.e., $f(\bar x)$ is the unique critical value of $f$ on $U.$ 
\end{proof}
 
Based on \cite{AbsilMahonyAndrews2005}, we now present some descent conditions ensuring the global convergence of iterates for smooth functions satisfying the PLK property. 
 
\begin{Proposition}\label{general convergence under KL}
Let $f:\R^n\rightarrow\R$ be a $\mathcal{C}^1$-smooth function, and let the following conditions hold along a sequence of iterates $\set{x^k}\subset\R^n$ for the function $f$:
\begin{itemize}
\item[\bf(H1)] {\rm(primary descent condition)}. There exists $\sigma>0$ such that for sufficiently large $k\in\N$, we have 
\begin{align*}
 f(x^k)-f(x^{k+1})\ge\sigma\norm{\nabla f(x^k)}\cdot\norm{x^{k+1}-x^k}.
\end{align*}
\item[\bf(H2)] {\rm(complementary descent condition)}. For sufficiently large $k\in\N$, we have
\begin{align*}
 \big[f(x^{k+1})=f(x^k)\big]\Longrightarrow [x^{k+1}=x^k].
\end{align*}
\end{itemize}
If $\bar x$ is an accumulation point of $\set{x^k}$  and $f$ satisfies the PLK property at $\bar x$, then $x^k\rightarrow\bar x$ as $k\to\infty$.
\end{Proposition}

The next proposition, which provides convergence rates for sequences of iterates under the PLK property from Definition~\ref{KL ine} with $\varphi(t) = Mt^{1 - q}$, $q\in(0,1)$, follows from \cite[Theorem~4]{bento25}, where the reader can find more references and discussions.

\begin{Proposition}\label{general rate}
Let $f:\R^n\rightarrow\R$ be a $\mathcal{C}^1$-smooth function, and let $\set{x^k}$ satisfy the conditions 
\begin{align}\label{two conditions}
f(x^k)-f(x^{k+1})\ge \alpha\norm{x^{k+1}-x^k}^2\;\text{ and }\;\norm{\nabla f(x^k)}\le \beta\norm{x^{k+1}-x^k}
\end{align}
for all $k\in\N$ sufficiently large with some constants $\alpha,\beta>0$. Suppose  in addition that $\bar x$ is an accumulation point of $\set{x^k}$ and that $f$ satisfies the $($exponential$)$ PLK property at $\bar x$ with $\varphi(t)=Mt^{1-q}$ for some $M>0$ and $q\in(0,1)$. Then the following convergence rates are guaranteed:

{\bf (i)} For $q\in(0,1/2)$, the sequence $\set{x^k}$ terminates at $\bar x$ in finite steps.

{\bf (ii)} For $q=1/2$, the sequence $\set{x^k}$ converges linearly to $\bar x$ as $k\to\infty$.

{\bf (iii)} For $q\in(1/2,1)$, we have  $\norm{x^k-\bar x}=\mathcal{O}( k^{-\frac{1-q}{2q-1}})$ as $k\to\infty$
\end{Proposition}

\begin{Remark}\rm\label{rmk general rate} It is easy to see  that condition \eqref{two conditions} in Proposition~\ref{general rate} yields the fulfillment of (H2) and of (H1) with $\sigma=\frac{\alpha}{\beta}$  in Proposition~\ref{general convergence under KL}. Proposition~\ref{general rate} obviously holds true even for extended-real-valued functions $f$ provided that $\{x^k\}$ is convergent and $f$ is ${\cal C}^1$-smooth on an open set containing $\{x^k\}$.
\end{Remark}

To proceed further, we recall the following fundamental construction that is usually employed in the analysis of momentum methods. Given a number $\alpha>0$, consider the {\em Lyapunov function} $H_\alpha:\R^n\times \R^n\rightarrow\R$ associated with the objective function $f$ that is defined by
\begin{align}\label{lyapunov func}
H_\alpha(x,y):=f(x)+\alpha\norm{x-y}^2\text{ for all }x,y\in\R^n.
\end{align}
To the best of our knowledge, function \eqref{lyapunov func} was first used by Zavriev and Kostyuk \cite{ZavrievKostyuk1993} in their study of convergence properties of Polyak's heavy ball method. More recently, this function has been widely utilized for the convergence analysis of exact first-order methods that incorporate momentum; see, e.g., \cite{Josz2023,WenChenPong2017} with more references and discussions therein.\vspace*{0.03in}

The next result, which follows from \cite[Theorem~3.3 and Corollary~3.5]{WangWang2023}, concerns the PLK property of the Lyapunov function $H_\alpha$ provided that the objective function $f$ satisfies this property. 

\begin{Proposition}\label{KL Lya global}
Let $f:\R^n\rightarrow\overline{\R}$ be ${\cal C}^1$-smooth around $\bar x$ and $\alpha>0.$ Then the following hold:
    
{\bf(i)} $H_\alpha$ satisfies the PLK property at $(\bar x,\bar x)$.
    
{\bf(ii)} If the PLK exponent of $f$ at $\bar x$ is $q\in [0,1)$, then the PLK exponent of the Lyapunov function  $H_\alpha$ at $(\bar x,\bar x)$ is $\max \set{q,1/2}.$
\end{Proposition}

Now we establish some uniformity in the selection of PLK desingularizing functions for a collection of Lyapunov functions $\set{H_\alpha}_{\alpha\ge \varepsilon}$ at $(\bar x,\bar x)$ provided that the objective function $f$ satisfies the PLK condition at the local minimizer $\bar x$. The result below is inspired by \cite[Proposition~3.7]{Josz2023}, where it is shown that this property holds under a global variant of the PLK condition and with $\varepsilon=1/4$.

\begin{Proposition}\label{isolated} Let $f:\R^n\rightarrow\overline\R$ be a ${\cal C}^1$-smooth function around its local minimizer $\bar x$. Assume that $f$ satisfies the PLK condition at $\bar x$ with the desingularizing function $\psi:[0,\eta)\rightarrow[0,\infty)$ for some $\eta\in (0,\infty]$. Then the following assertions hold whenever $\varepsilon\in (0,1/4]$:
   
{\bf(i)} The collection of functions $\set{H_\alpha}_{\alpha\ge \varepsilon}$ satisfies the same PLK condition at $\bar z:=(\bar x,\bar x)$ with the desingularizing function $\varphi:[0,\eta)\rightarrow[0,\infty)$ defined by $\varphi(0):=0$ and 
\begin{align}\label{construction varphi}
\varphi(t):=\sqrt{\frac{8}{\varepsilon}}\int_0^{t/2} \max\set{\psi'(s),\frac{1}{\sqrt{s}}}ds\;\text{ for all }\;t\in (0,\eta).
\end{align}
        
{\bf(ii)} If $\psi(t)=Mt^{1-q}$ for some $M>0$ and $q\in[0,1)$, then the function $\varphi$ in \eqref{construction varphi} is given by $\varphi(t)=\frac{2^{\theta+1/2}\tilde M}{\sqrt{\varepsilon}(1-\theta)}t^{1-\theta}$ when $t\in (0,\delta)$ for some $\delta\in(0,\eta)$, where $\theta:=\max\set{q,\frac{1}{2}}$ and $\tilde{M}>0$ is defined as 
\begin{align}\label{tilde M}
\tilde{M}:=\begin{cases}
1&\text{ if }\quad q\in \left[0,\frac{1}{2}\right),\\
\max \set{M(1-q),1}&\text{ if }\quad q=\frac{1}{2},\\
M(1-q)&\text{ if }\quad q\in \left(\frac{1}{2},1\right).
\end{cases}
\end{align}
\end{Proposition}
\begin{proof}
To verify (i) pick any $\varepsilon \in (0,1/4]$ and use the PLK property of $f$ at its local minimizer $\bar x$ to find a neighborhood $U$ of $\bar x$ on which $f$ satisfies all the conditions in Definition~\ref{KL ine} with the corresponding desingularizing function $\psi: [0,\eta) \rightarrow [0,\infty)$. In particular, we have
\begin{align}\label{desingularizing}
\psi'\big(f(x)-f(\bar{x})\big)\norm{\nabla f(x)}\ge 1\;\mbox{  for all }\;x \in U\;\mbox{ with }\;f(\bar{x}) < f(x) < f(\bar{x}) + \eta.
\end{align}
The function $\varphi$ in \eqref{construction varphi} can be written as $\varphi(t)=\sqrt{\frac{8}{\varepsilon}}\phi\brac{\frac{t}{2}}$, where $\phi:[0,\eta)\rightarrow[0,\infty)$ is defined by 
\begin{align}\label{construction phi t}
\phi(t):=\int_0^t \max\set{\psi'(s),\frac{1}{\sqrt{s}}}ds\;\text{ for }\;t\in(0,\eta)\;\mbox{ with }\;\phi(0):=0.
\end{align}
It follows from the continuity of $\psi$ at $0$ that
\begin{align*}
0\le \phi(t)\le \int_0^t \psi'(s)ds+ \int_0^t\frac{1}{\sqrt{s}}ds=\psi(t)+2\sqrt{t}\;\text{ for }\;t\in(0,\eta).
\end{align*}
Passing the limit above as $t\rightarrow0^+$ and using the continuity of $\psi$ reveal that $\phi$ is also continuous at $0$. Moreover, we deduce from \eqref{construction phi t} that the derivative $\phi'(t)=\max \set{\psi'(t),\frac{1}{\sqrt{t}}}>0$ is positive, continuous, and nonincreasing on $(0,\eta)$, which ensures the concavity of $\phi$ on this interval. In addition, it follows that 
\begin{align}\label{phi'>1}
\phi'(f(x)-f(\bar x))\norm{\nabla f(x)}\ge \psi'(f(x)-f(\bar x))\norm{\nabla f(x)}\ge 1
\end{align}
for $x\in U$ with $f(\bar x)< f(x)<f(x)+\eta$. Since $\varphi(t)=\sqrt{\frac{8}{\varepsilon}}\phi\brac{\frac{t}{2}}$, the domain of $\varphi$ is $[0,2\eta).$ It is clear that $\varphi$ satisfies all the conditions as a desingularizing function on $[0,2\eta)$ with
\begin{align}\label{deri varphi}
\varphi'(t)=\sqrt{\frac{2}{\varepsilon}}\phi'\brac{\frac{t}{2}}>0 \text{ for all }t\in (0,2\eta).
\end{align}
Taking any $\alpha\ge \varepsilon$ and $z=(x,y)\in Z:= U\times U$ with 
\begin{align}\label{halpha}
H_\alpha(\bar z)<H_\alpha(z)<H_\alpha(\bar z)+\eta,
\end{align}
we aim at verifying the inequality
\begin{align}\label{KL R2n}
\varphi'(H_\alpha(z)-H_\alpha(\bar z))\norm{\nabla H_\alpha (z)}\ge 1.
\end{align}
To proceed, observe first the estimates
\begin{subequations}
\begin{align}
\norm{\nabla H_\alpha(x,y)}&=\sqrt{\norm{\nabla f(x)+2\alpha(x-y)}^2+\norm{2\alpha(x-y)}^2} \label{g_Hal.a}\\
&\ge \sqrt{\frac{1}{4}\norm{\nabla f(x)}^2+\alpha^2\norm{x-y}^2}\label{g_Hal.b}\\
&\ge \sqrt{\frac{1}{4}\norm{\nabla f(x)}^2+\alpha\varepsilon\norm{x-y}^2}\label{g_Hal.c}\\
&\ge \sqrt{\frac{\varepsilon}{2}}\brac{\norm{\nabla f(x)}+\sqrt{\alpha}\norm{x-y}},\label{g_Hal.d}
\end{align}
\end{subequations}
where \eqref{g_Hal.a} is comes from the construction of $H_\alpha$, \eqref{g_Hal.b} follows from the inequality $\norm{a+b}^2\ge \frac{1}{4}\norm{a}^2-\frac{3}{4}\norm{b}^2$ for $a,b\in\R^n$, \eqref{g_Hal.c} follows from $\alpha\ge \varepsilon$, and \eqref{g_Hal.d} follows from $a^2+b^2\ge \frac{1}{2}(a+b)^2$ for $a,b\in\R$ and $\varepsilon\le 1/4$.  Now let $S:=\max \set{\alpha\norm{x-y}^2,f(x)-f(\bar x)}$ and deduce from $f(x)\ge f(\bar x)$ as $x\in U$ that
\begin{align*}
H_\alpha(z)-H_\alpha(\bar z)=f(x)+\alpha\norm{x-y}^2-f(\bar x) \in [S,2S].
\end{align*}
Combining the latter with $0< H_\alpha(z)-H_\alpha(\bar z)<\eta$ by \eqref{halpha} gives us $S\in (0,\eta).$ Using the nonincreasing property of $\varphi'$ and the relation $\varphi'(t)=\sqrt{\frac{2}{\varepsilon}}\phi'\brac{\frac{t}{2}}$ from \eqref{deri varphi} yields
\begin{align*}
\varphi'(H_\alpha(z)-H_\alpha(\bar z))&\ge \varphi'\brac{2S}=\sqrt{\frac{2}{\varepsilon}}\phi'(S).
\end{align*}
Furthermore, it follows from  \eqref{deri varphi} and  \eqref{g_Hal.d} that
\begin{align*}
\varphi'(H_\alpha(z)-H_\alpha(\bar z))\norm{\nabla H_\alpha(x,y)}&\ge \sqrt{\frac{2}{\varepsilon}}\phi'(S)\sqrt{\frac{\varepsilon}{2}}\brac{\norm{\nabla f(x)}+\sqrt{\alpha}\norm{x-y}}\\
&=\phi'(S)\norm{\nabla f(x)}+\phi'(S)\sqrt{\alpha}\norm{x-y}.
\end{align*}
By $\phi'\big(f(x)-f(\bar{x})\big)\norm{\nabla f(x)}\ge 1$ coming from \eqref{phi'>1} and $\phi'(\alpha \norm{x-y}^2)\ge \frac{1}{\sqrt{\alpha}\norm{x-y}}$ due to $\phi'(t)\ge \frac{1}{\sqrt{t}}$ for $t\in (0,\eta),$ we get the estimate
\begin{align*}
\phi'(S)\norm{\nabla f(x)}+\phi'(S)\sqrt{\alpha}\norm{x-y}\ge 1,
\end{align*}
which therefore verifies assertion (i) of the proposition.

To prove assertion (ii)  is easier. Indeed, let $\psi(t)=Mt^{1-q}$ for $q\in[0,1)$, let $\theta:=\max\set{\frac{1}{2},q}\in (0,1)$, and take $\tilde M$ from \eqref{tilde M}. It readily follows that
\begin{align}\label{formu phi}
\max\set{\psi'(s),\frac{1}{\sqrt{s}}}=\tilde Ms^{-\theta}\text{ for all }s\in(0,\delta)
\end{align}
for some $\delta\in(0,\eta)$ sufficiently small.
Combining the latter with \eqref{construction phi t}, we arrive at 
\begin{align*}
 \varphi(t)=\sqrt{\frac{8}{\varepsilon}}\int_0^{t/2}\tilde{M}s^{-\theta}ds=\frac{2^{\theta+1/2}\tilde M}{\sqrt{\varepsilon}(1-\theta)}t^{1-\theta} \text{ for all }t\in (0,\delta),
\end{align*}
which thus verifies (ii) and completes the proof of the proposition.
\end{proof} 

We conclude this section with the following simple albeit useful result taken from \cite[Proposition~B.5]{KhanhLuongMordukhovichTran}.

\begin{Proposition}\label{convergence rate deduce}
Let $f:\Omega\rightarrow{\R}$ be a function defined on open and convex set of $\Omega\subset\R^n$, and let $\bar x \in\Omega$. Assume that $f$ has a Lipschitz continuous gradient on $\Omega$ and take a sequence $\set{x^k}\subset\Omega $ converging to $\bar x$. Suppose in addition that there exists $\alpha>0$ such that 
\begin{align}\label{main condition rate}
\alpha\norm{\nabla f(x^k)}^2\le f(x^k)-f(x^{k+1})\;\text{ for all large }\;k\in\N.
\end{align}
Then the linear convergence of $x^k\rightarrow\bar x$ ensures the linear convergences of $f(x^k)$ to $f(\bar x)$ and $\nabla f(x^k)$ to $0$ as $k\to\infty$. Moreover, the convergence rate $\norm{x^k-\bar x}=\mathcal{O}(m(k))$ with $m(k)\downarrow 0$ yields
$$
\abs{f(x^k)-f(\bar x)}=\mathcal{O}(m^2(k))\;\mbox{ and }\;\norm{\nabla f(x^k)}=\mathcal{O}(m(k))\;\mbox{ as }\;
k\to\infty. 
$$
\end{Proposition}\vspace*{-0.2in}

\section{Inexact Gradient Descent Method with Momentum}\label{sec: IGDm}\vspace*{-0.05in}

In this section, we aim to introduce the inexact gradient descent method with momentum (IGDm) and conduct its global and local convergence analysis. The section is split into two parts. In the first subsection, we design the basic algorithm of IGDm, reveal some of its major properties, and compare it with other algorithms involving momentum. The second subsection develops a detailed global and local analysis of IGDm with deriving constructive convergence rates under the PLK conditions.\vspace*{-0.1in}

\subsection{Algorithm Design and Some Properties}\label{subsect convergence nanalys}

Here is the basic algorithm of IGDm to solve problem \eqref{main problem} illustrated in Figure~1.\vspace*{0.05in}

\begin{longfbox}
\begin{Algorithm}[IGDm]\hlabel{IGDm}\quad
\begin{enumerate}[-]
\item \textbf{Optimization problem: } 
\begin{align*}
\eqref{main problem}\;\text{ with }\begin{cases}X=\R^n\;\;\text{ for global minima,}\\
X=\B(\bar x,r)\;\text{ for a local minimizer }\bar x\;\text{ with some }\;r>0.
\end{cases}
\end{align*}
\vspace{-0.5cm}
\item \textbf{Initialization: }Choose  $x^0=x^1\in\Omega,\;\set{\beta_k}\subset [0,\infty),\;\set{\gamma_k}\subset[0,\infty)$,  $\nu \in(0,1),$ $\tau>0$.
\item \textbf{Parameter conditions:} $\bar{\beta}:=\sup \beta_k<1,$ $\bar \delta:=\sup|\beta_k-\gamma_k|<\infty.$
\item \textbf{Iteration: } $(k\ge 1)$ Update:
\begin{align}
&x^k_{\rm in}:=x^k+\beta_k(x^k-x^{k-1}),\label{step xin}\\
&x^k_{\rm ex}:=x^k+\gamma_k(x^k-x^{k-1}),\label{step xex}\\
&x^{k+1}:=x^k_{\rm in}-\tau  g^k,\;\text{ where }\;\norm{g^k-\nabla f(x^k_\rex)}\le \nu \norm{g^k}.\label{iteration update}
\end{align}
\end{enumerate}
\end{Algorithm}
\end{longfbox}
\vspace{0.3cm}

The following remark discusses the momentum effect in the algorithm. 
\begin{Remark}\rm\label{rmk mometum} 
 \begin{figure}[H]
\centering
\includegraphics[scale=0.3]{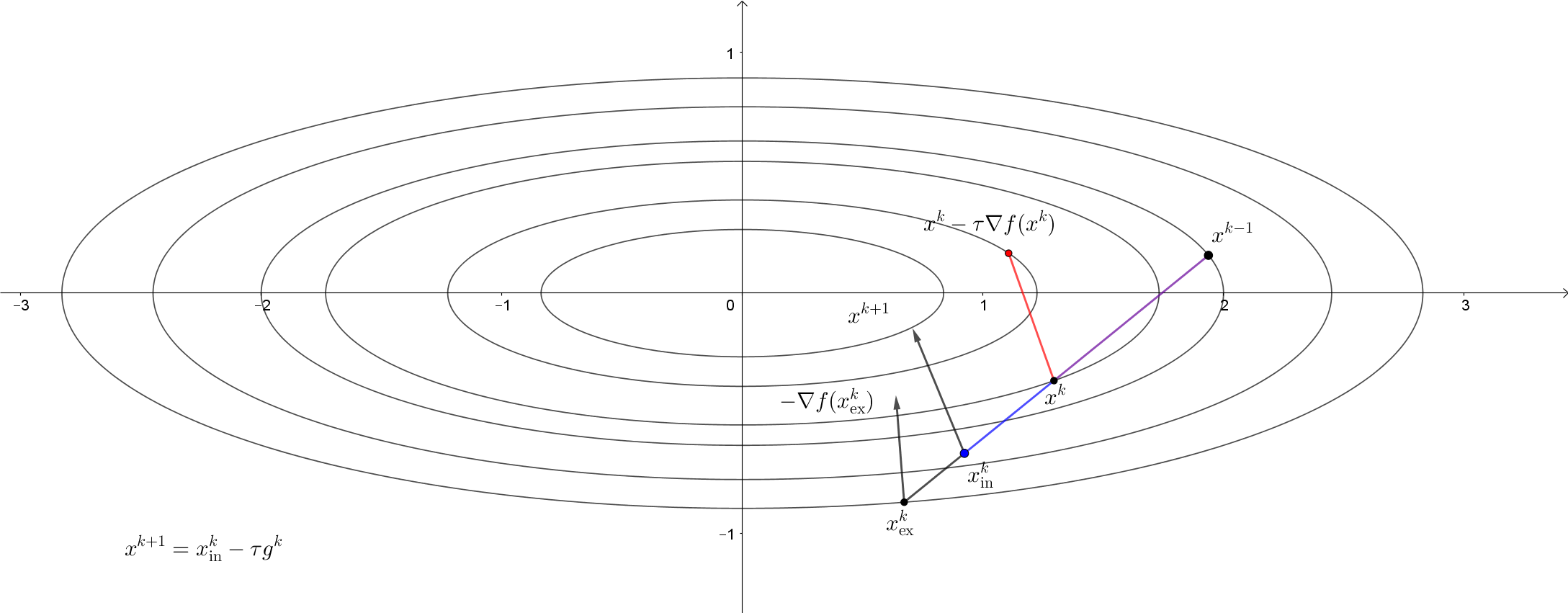}
\caption{Inexact Gradient Descent Algorithm with Momentum }
\label{fig:enter-label}
\end{figure}
The inertial update for $x^{k}_{\rm in}$ uses the momentum term as the starting point for the next iterative update. In contrast, the extrapolation update for $x^{k}_{\rm ex}$ incorporates momentum into the term for which the gradient is calculated. In the exact case, i.e., when $g^k=\nabla f(x^k_\rex)$ in Step~2,  Algorithm~\ref{IGDm} encompasses various gradient descent methods with momentum commonly used in practice. The differences between those methods are distinguished by the selections of $\set{\beta_k}$ and $\set{\gamma_k}$ as follows:
\begin{enumerate}[\bf (i)]

\item Polyak's heavy ball method \cite{Polyak1964,ZavrievKostyuk1993}: $\beta_k=\frac{4}{\brac{\sqrt{L}+\sqrt{\mu}}^2}$ and  $\gamma_k=0$ for all $k\in\N$, where  $\mu$ is a strong convexity modulus of $f$ and $L$ is a Lipschitz modulus of $\nabla f$. 

\item Nesterov's acceleration for strongly convex functions \cite[Equations~(2.2.8), (2.2.22)]{Nesterov2018}: $\beta_k=\gamma_k:=\frac{\sqrt{L}-\sqrt{\mu}}{\sqrt{L}+\sqrt{\mu}}\text{ for all }k\in\N$, where $\mu$ and $L$ are as in (i).

\item Nesterov's acceleration for convex functions (see \cite{Nesterov1983} and \cite[Eqn.~1.5]{ShiDuJordanSu2022}): $\beta_k=\gamma_k=\frac{k}{k+3}$, $k\in\N.$

\item FISTA for convex functions \cite{BeckTeboulle2009}: $\beta_k=\gamma_k=\frac{\theta_{k-1}-1}{\theta_k},$ where $\theta_0=\theta_1=1$, $\theta_{k+1}=\frac{1+\sqrt{1+4\theta^2_k}}{2}$ as $k\ge 1.$
\end{enumerate}
\end{Remark}

Define further the sequence $\{z^k\} \subset \mathbb{R}^{n}\times \mathbb{R}^n$ by $z^{k}:= (x^{k},x^{k-1})$ for all $k\in\N$ along which the fundamental properties of the Lyapunov function $H_\alpha$ are establishes below. These properties are crucial for deriving and unifying local and global convergence properties of Algorithm~\ref{IGDm}.

\begin{Proposition}\label{prop AIGD tech}
Let $\set{x^k}$ be a sequence generated by Algorithm~{\rm\ref{IGDm}} with 
\begin{align}\label{max Lt<1-nu}
\max\set{L\tau,2L\tau \bar \delta+(L\tau+1)\bar \beta^2}< 1-\nu
\end{align}
such that $x^k,x^k_{\rm in},x^k_{\rm ex}\in X$ for all $k=1,\ldots,K\in\N$. Consider the positive constants
\begin{align}\label{C1C2}
\alpha:&=\frac{(L \tau+1)\bar \beta^2+1-\nu }{4\tau},\nonumber\\
C_1:&=\frac{1-\nu-(L \tau+1)\bar \beta^2-2L\tau \bar \delta}{4\tau},\\
C_2:&=\sqrt{2}\max\set{\frac{\nu+1}{\tau },\frac{\bar\beta(L\tau+\nu+1)+L\tau\bar\delta}{\tau }}+4{\alpha}.\nonumber
\end{align}
Then we have the following estimates whenever $k=1,\ldots,K-1:$
\begin{align}
C_1\norm{z^{k+1}-z^k}^2&\le H_{\alpha}(z^k)-H_{\alpha}(z^{k+1}),\label{3.2(i)}\\
\norm{\nabla H_{\alpha}(z^{k})}&\le C_2\norm{z^{k+1}-z^k}\label{3.2(ii)}.
\end{align}
\end{Proposition}
\begin{proof}
By Lemma~\ref{lemma descent}, the Lipschitz continuity with constant $L>0$ of $\nabla f$ on $X $ yields
\begin{align*}
-\frac{L}{2}\norm{y-x}^2\le f(y)-f(x)-\dotproduct{\nabla f(x),y-x}\le \frac{L}{2}\norm{y-x}^2 \text{ for all }x,y\in X .
\end{align*}
For any $k=1,\ldots,K-1$, we use  the left-hand side inequality above with $y=x^k,x=x^k_{\rm in}$ and the right-hand side inequality with $y=x^{k+1},x=x^k_{\rm in}$ to get
\begin{align*}
 f(x^k)-f(x^k_{\rm in})-\dotproduct{\nabla f(x^k_{\rm in}),x^k-x^k_{\rm in}}&\ge -\frac{L }{2}\norm{x^k-x^k_{\rm in}}^2,
 \\
 f(x^{k+1})- f(x^k_{\rm in}) -\dotproduct{\nabla f(x^k_{\rm in}),x^{k+1}-x^k_{\rm in}}&\le \frac{L}{2}\norm{x^{k+1}-x^k_{\rm in}}^2.
\end{align*}
Subtracting the first inequality from the second one gives us 
\begin{align*}
f(x^{k+1})-f(x^k)-\dotproduct{\nabla f(x^k_{\rm in}),x^{k+1}-x^k}\le \frac{L}{2}\norm{x^{k+1}-x^k_{\rm in}}^2+\frac{L }{2}\norm{x^k-x^k_{\rm in}}^2,
\end{align*}
which can be rewritten in the form
\begin{align}\label{prop: 1 chain}
\begin{array}{ll}
\disp f(x^{k+1})-f(x^k)\le \frac{L}{2}\norm{x^{k+1}-x^k_{\rm in}}^2+\frac{L }{2}\norm{x^k-x^k_{\rm in}}^2+\dotproduct{\nabla f(x^k_{\rm in}),x^{k+1}-x^k}\\
=\disp\frac{L}{2} \norm{x^{k+1}-x^k_{\rm in}}^2+\frac{L \beta_k^2}{2}\norm{x^k-x^{k-1}}^2+\dotproduct{\nabla f(x^k_{\rm in})-\nabla f(x^k_{\rm ex}),x^{k+1}-x^k}\\
\disp+\dotproduct{\nabla f(x^k_{\rm ex})- g^k,x^{k+1}-x^k}+\dotproduct{g^k,x^{k+1}-x^k}.
\end{array}
\end{align}
By the Cauchy-Schwarz inequality, the Lipschitz continuity of $\nabla f$ on $X $, and the constructions of $x^k_{\rm in}$ and $x^k_{\rm ex}$ in \eqref{step xin} and \eqref{step xex}, we deduce the estimates
\begin{align}\label{prop: 1.1 chain}
\begin{array}{ll}
\dotproduct{\nabla f(x^k_{\rm in})-\nabla f(x^k_{\rm ex}),x^{k+1}-x^k}\le \norm{\nabla f(x^k_{\rm in})-\nabla f(x^k_{\rm ex})}\cdot\norm{x^{k+1}-x^k}\\
\le L\norm{x^k_{\rm in}-x^k_{\rm ex}}\cdot\norm{x^{k+1}-x^k}\le L\abs{\beta_k-\gamma_k}\norm{x^k-x^{k-1}}\cdot\norm{x^{k+1}-x^k}\\
\le\disp\frac{L\bar \delta}{2}\brac{\norm{x^k-x^{k-1}}^2+\norm{x^{k+1}-x^k}^2}.
\end{array}
\end{align}
It follows from the Cauchy-Schwarz and AM-GM inequalities together with the algorithm design $x^{k+1}:=x^k_{\rm in}-\tau  g^k$ and $\norm{g^k-\nabla f(x^k_{\rm ex})}\le \nu\norm{g^k}$ in \eqref{iteration update} that
\begin{align}\label{prop: 3 chain}
\begin{array}{ll}
\dotproduct{\nabla f(x^k_{\rm ex})- g^k,x^{k+1}-x^k}\le \norm{\nabla f(x^k_{\rm ex})- g^k}\cdot\norm{x^{k+1}-x^k}\\
\disp\le \nu \norm{g^k}\cdot\norm{x^{k+1}-x^k}=\frac{\nu}{\tau }\norm{x^k_{\rm in}-x^{k+1}}\cdot\norm{x^{k+1}-x^k}\\
\le\disp \frac{\nu}{2\tau }\brac{\norm{x^{k+1}-x^k_{\rm in}}^2+\norm{x^{k+1}-x^k}^2}.
\end{array}
\end{align}
Using further \eqref{step xin} and \eqref{iteration update},  we arrive at
\begin{align}\label{prop: 2 chain}
\begin{array}{ll}
\dotproduct{g^k,x^{k+1}-x^k}=\disp\frac{1}{\tau }\dotproduct{x^k_{\rm in}-x^{k+1},x^{k+1}-x^k}\\
=\disp\frac{1}{2\tau }\brac{\norm{x^k_{\rm in}-x^k}^2-\norm{x^{k+1}-x^k_{\rm in}}^2-\norm{x^k-x^{k+1}}^2}\\
=\disp\frac{1}{2\tau }\brac{\beta_k^2\norm{x^{k}-x^{k-1}}^2-\norm{x^{k+1}-x^k_{\rm in}}^2-\norm{x^k-x^{k+1}}^2}.
\end{array}
\end{align}
Combining \eqref{prop: 1 chain}--\eqref{prop: 2 chain} with taking into account $L\tau +\nu<1$ and $\beta_k\le \bar \beta$ gives us
\begin{align}\label{estimane conse}
\begin{array}{ll}
f(x^{k+1})-f(x^k)&\le\disp\frac{L\tau +\nu-1}{2\tau }\norm{x^{k+1}-x^k_{\rm in}}^2+\frac{L\tau\bar \delta+\nu-1}{2\tau }\norm{x^{k+1}-x^k}^2\\
&+\disp\frac{L \tau \beta_k^2+\beta_k^2+L\tau\bar \delta}{2\tau }\norm{x^k-x^{k-1}}^2\\
&\le\disp\frac{L\tau\bar \delta+\nu-1}{2\tau }\norm{x^{k+1}-x^k}^2+\frac{(L \tau +1)\bar \beta^2+L\tau\bar \delta}{2\tau }\norm{x^k-x^{k-1}}^2.
\end{array}
\end{align}
Since $\bar\beta^2<\frac{1-\nu}{L \tau+1}$, it follows that $\frac{(L \tau+1)\bar \beta^2+L\tau\bar \delta}{2\tau }<\frac{1-\nu-L\tau\bar \delta}{2\tau }$, and by $\alpha=\frac{(L \tau+1)\bar \beta^2+1-\nu}{4\tau}$ we get 
\begin{align*}
C_1=\frac{1-\nu-(L \tau+1)\bar \beta^2-2L\tau \bar \delta}{4\tau}=\frac{1-\nu-L\tau\bar\delta}{2\tau}-\alpha=\alpha-\frac{(L  \tau+1)\bar\beta^2+L\tau\bar \delta}{2\tau}
\end{align*}
for $C_1$ in \eqref{C1C2}. Using the latter together with \eqref{estimane conse} and the construction of $H_\alpha$ in \eqref{lyapunov func}  brings us to
\begin{align*}
H_{\alpha}(z^{k+1})-H_{\alpha}(z^k)&= f(x^{k+1})+\alpha\norm{x^{k+1}-x^k}^2-f(x^{k})-\alpha\norm{x^{k}-x^{k-1}}^2\\
&\le\brac{ \frac{\nu-1+L\tau\bar \delta}{2\tau }+\alpha}\norm{x^{k+1}-x^k}^2+\brac{\frac{(L \tau+1)\bar  \beta^2+L\tau\bar \delta}{2\tau }-\alpha}\norm{x^k-x^{k-1}}^2\\
&\le- C_1\brac{\norm{x^{k+1}-x^k}^2+\norm{x^k-x^{k-1}}^2}=-C_1\norm{z^{k+1}-z^k}^2,
\end{align*}
which therefore justifies the inequality in \eqref{3.2(i)}. 

To verify \eqref{3.2(ii)}, take any $k=1,\ldots,K-1$ and obtain the relationships
\begin{subequations}
\begin{align}
\norm{\nabla f(x^k)}&\le \norm{\nabla f(x^k)-\nabla f(x^k_{\rm in})} +\norm{\nabla f(x^k_{\rm in})-\nabla f(x^k_{\rm ex})}+\norm{\nabla f(x^k_{\rm ex})-g^k}+\norm{g^k}\label{nab.1 triang}\\
&\le(L\bar \beta+L\bar\delta)\norm{x^k-x^{k-1}}+(\nu+1)\norm{g^k}\label{nab.2}\\
&=(L\bar \beta+L\bar\delta)\norm{x^k-x^{k-1}}+\frac{\nu+1}{\tau }\norm{x^k+\beta_k(x^k-x^{k-1})-x^{k+1}}\label{nab.3}\\
&\le \brac{L\bar \beta+L\bar\delta+\frac{\nu+1}{\tau }\bar \beta}\norm{x^k-x^{k-1}}+\frac{\nu+1}{\tau }\norm{x^{k+1}-x^k}\nonumber\\
&\le \max\set{\frac{\nu+1}{\tau },L\bar \beta+L\bar\delta+\frac{\nu+1}{\tau }\bar \beta}\brac{\norm{x^{k+1}-x^k}+\norm{x^k-x^{k-1}}}\nonumber\\
&\le \sqrt{2}\max\set{\frac{\nu+1}{\tau },L\bar \beta+L\bar\delta+\frac{\nu+1}{\tau }\bar \beta}\norm{z^{k+1}-z^k}\label{nab.6},
\end{align}
\end{subequations}
where \eqref{nab.1 triang} follows from the triangle inequality; \eqref{nab.2} follows from $\norm{g^k-\nabla f(x^k_{\rm ex})}\le \nu\norm{g^k}$ in \eqref{iteration update}, the updates in \eqref{step xin}, \eqref{step xex}, and the Lipschitz continuity of $\nabla f$ on $X $ containing $ x^k_\rin,x^k_\rex$; \eqref{nab.3} follows from \eqref{step xin} and \eqref{step xex}; \eqref{nab.6} follows from $a+b\le \sqrt{2(a^2+b^2)}$ and $z^k=(x^k,x^{k-1})$. By using the estimates above, we deduce from the construction of $H_\alpha$ and inequality $\sqrt{a^2+b^2}\le a+b$ for $a,b\ge 0$ that
\begin{align*}
\norm{\nabla H_{\alpha}(z^k)}&= \sqrt{\norm{\nabla f(x^k)+2{\alpha}(x^k-x^{k-1})}^2+\norm{2{\alpha} (x^k-x^{k-1})}^2}\\
&\le \norm{\nabla f(x^k)+2{\alpha}(x^k-x^{k-1})}+\norm{2{\alpha} (x^k-x^{k-1})}\\
& \le \norm{\nabla f(x^k)}+4{\alpha}\norm{x^k-x^{k-1}}\\
&\le \brac{\sqrt{2}\max\set{\frac{\nu+1}{\tau },L\bar \beta+L\bar\delta+\frac{\nu+1}{\tau }\bar \beta}+4{\alpha}}\norm{z^{k+1}-z^k},
\end{align*}
which justifies \eqref{3.2(ii)}  and thus completes the proof of the proposition.
\end{proof}

The next proposition reveals the relationships between the subsequent iterations of Algorithm~\ref{IGDm} and the sequence of $z^k=(x^k,x^{k-1})$, $k\in\N$, defined above.

\begin{Proposition}\label{proposition distance}
Let $\set{x^k}$ be the sequence of iterates generated by Algorithm~{\rm\ref{IGDm}} such that $x^k,x^k_{\rm ex}\in \B(\bar x,\rho)\subset X$ for all $k=1,\ldots,K$. Then we have for all $k=1,\ldots,K$ that
\begin{align}
\norm{x^{k+1}-x^k}&\le   \frac{L\tau\rho }{(1-\bar \beta)(1-\nu)},\label{3.3(i)}\\
\norm{z^{k+1}-z^k}&\le  \frac{\sqrt{2} L\tau\rho }{(1-\bar \beta)(1-\nu)}.\label{3.3(ii)}
\end{align}
\end{Proposition}
\begin{proof}
Pick any $k=1,\ldots,K$ and observe that the iteration update in Algorithm~\ref{IGDm} can be rewritten as $x^{i+1}=x^i+\beta_k(x^i-x^{i-1})-\tau g^i$, which yields
\begin{align}\label{new update}
\norm{x^{i+1}-x^i}\le \bar \beta \norm{x^i-x^{i-1}}+\tau\norm{g^i}\;\text{ for all }\;i\le k.
\end{align}
Using the triangle inequality together with $\norm{g^i-\nabla f(x^i_{\rm ex})}\le \nu \norm{g^k}$ gives us
\begin{align*}
\norm{\nabla f(x^i_{\rm ex})}\ge \norm{g^i}-\norm{\nabla f(x^i_{\rm ex})-g^i}\ge (1-\nu)\norm{g^i}\;\text{ for all }\;i\le k.
\end{align*} 
Proceeding inductively from \eqref{new update} for $i=k,k-1,\ldots,1$ with the usage of the Lipschitz continuity of $\nabla f$ on the ball $\B(\bar x,\rho)$ containing $\bar x$ and $x^{k-i}_\rex,i=0,\ldots,k-1$ yields
\begin{align*}
\norm{x^{k+1}-x^k}&\le \bar \beta\norm{x^k-x^{k-1}}+\tau \norm{g^k}\\
&=\bar \beta^2\norm{x^{k-1}-x^{k-2}}+\tau\brac{\bar \beta \norm{g^{k-1}}+\norm{g^k}}\\
&\ldots\\
&\le\bar \beta^{k}\norm{x^1-x^0}+\tau \sum_{i=0}^{k-1}\bar \beta^i\norm{g^{k-i}}\\
&\le \frac{\tau}{1-\nu} \sum_{i=0}^{k-1}\bar \beta^i\norm{\nabla f(x_\rex^{k-i})} = \frac{\tau}{1-\nu}\sum_{i=0}^{k-1}\bar \beta^i\norm{\nabla f(x^{k-i}_\rex)-\nabla f(\bar x)}\\
&\le \frac{\tau}{1-\nu}\sum_{i=0}^{k-1}\bar \beta^i L\norm{x^{k-i}_\rex -\bar x}\le \frac{\tau}{1-\nu}\sum_{i=0}^{k-1}\bar \beta^i L\rho\\
&\le \frac{L\tau\rho}{1-\nu}\sum_{i=0}^{\infty}\bar \beta^i = \frac{L\tau\rho}{(1-\nu)(1-\bar \beta)},
\end{align*}
where the latter inequality follows from $\bar \beta\in [0,1)$. This readily verifies \eqref{3.3(i)}. Combining \eqref{3.3(i)}  with $\norm{z^{k+1}-z^k}=\sqrt{\norm{x^{k+1}-x^k}^2+\norm{x^{k}-x^{k-1}}^2}$ brings us to \eqref{3.3(ii)} and thus completes the proof.
\end{proof}

Finally, we show that the limit of iterates in Algorithm~\ref{IGDm} is a stationary point of problem \eqref{main problem}.

\begin{Proposition}\label{convergence xk+1 xk} Let $\set{x^k}$ be the sequence of iterates generated by Algorithm~{\rm\ref{IGDm}} that converges to some $\tilde x\in X$. Then $\tilde x$ is a stationary point of \eqref{main problem} and the sequences $\set{x^k_\rex},\set{x^k_\rin}$ converge to $\tilde x$ as well.
\end{Proposition}
\begin{proof} Since the sequence $\set{x^k}$ is convergent,  we have $\norm{x^{k+1}-x^k}\rightarrow 0$ as $k\rightarrow\infty.$ Combining this with the updates in \eqref{step xin} and \eqref{step xex} together with the boundedness of $\set{\beta_k} and \set{\gamma_k}$ tells us that
\begin{align*}
{x^k_{\rm in}}&=x^k+\beta_k({x^{k}-x^{k-1}})\rightarrow\tilde x,\\
{x^k_{\rm ex}}&=x^k+\gamma_k({x^{k}-x^{k-1}})\rightarrow\tilde x.
\end{align*}
Using $\norm{x^{k+1}-x^k}\rightarrow 0$ and the updates in \eqref{step xin}, \eqref{iteration update} of Algorithm~\ref{IGDm} yields
\begin{align*}
g^k=\frac{1}{\tau}\brac{x^k-x^{k+1}+\beta_k(x^k-x^{k-1})}\rightarrow 0\;\text{ as }\;k\rightarrow\infty.
\end{align*}
It follows from $\norm{\nabla f(x^k_{\rm ex})-g^k}\le \nu\norm{g^k}$ in \eqref{iteration update} that ${\nabla f(x^k_{\rm ex})}\rightarrow 0$, and we are done with the proof.
\end{proof}\vspace*{-0.15in}

\subsection{Convergence Analysis}\label{subsec: convergence ana}

This subsection presents two major results about the global and local convergence of iterates in Algorithm~\ref{IGDm} to a stationary point of \eqref{main problem}  with establishing convergence rates under the PLK conditions. Since the algorithm obviously stops when a stationary point is reached, in our convergence analysis we impose the following natural requirement:

\begin{Assumption}\rm\label{assu}
$\nabla f(x^k_{\rm ex})\ne 0$ for all $k\in\N$.
\end{Assumption}

The theorem below establishes the {\em global convergence} of IGDm with providing its {\em convergence rates} under the exponential PLK conditions from Definition~\ref{KL ine}. Note that the lower exponent case $q\in(0,1/2)$ yielding the finite termination by Proposition~\ref{general rate}(i), does not have any effect in what follows due to the result of Proposition~\ref{KL Lya global}(ii) for the Lyapunov function used here in the study of methods with {\em momentum}. 

\begin{Theorem}\label{global IGDm}
Let $\set{x^k}$ be a sequence of iterates generated by Algorithm~{\em\ref{IGDm}} with $X=\R^n$, and let 
\begin{equation}\label{glob-assum}
\max\set{L\tau,2L\tau \bar \delta+(L\tau+1)\bar \beta^2}< 1-\nu.
\end{equation}
If $\inf_{k\in\N} f(x^k)>-\infty$, then we have the  assertions:
\begin{enumerate}[\bf(i)]
\item The sequence $\set{\nabla f(x^k)}$ converges to $0$ as $k\rightarrow\infty.$

\item If the basic PLK condition holds for $f$ at an accumulation point $\bar x$ of $\set{x^k}$, then $\set{x^k}$ converges to $\bar x$, which is a stationary point of $f$.

\item Assume that the exponential PLK condition holds for $f$ at $\bar x$ with the desingularizing function $\psi(t)=Mt^{1-q}$, where $M>0$ and $q\in[0,1)$. Then we have the convergence rates:

$\bullet$ For $q=[0,1/2]$, the sequences $\set{x^k}$, $\set{\nabla f(x^k)}$, and $\{f(x^k)\}$ converge linearly to the values $\bar x$, $0$, and $f(\bar x),$ respectively.
   
$\bullet$ For $q\in(1/2,1)$, the corresponding convergence rates are 
$$
\norm{x^k-\bar x}=\mathcal{O}\brac{k^{-\frac{1-q}{2q-1}}},\quad\norm{\nabla f(x^k)}=\mathcal{O}\brac{k^{-\frac{1-q}{2q-1}}},\quad
f(x^k)-f(\bar x)=\mathcal{O}\brac{k^{-\frac{2-2q}{2q-1}}}.
$$
\end{enumerate}
\end{Theorem}
\begin{proof}
To verify (i), pick an accumulation point $\bar x$ of $\set{x^k}$ and observe by $X=\R^n$ and \eqref{glob-assum} that all the assumptions of Proposition~\ref{prop AIGD tech} are satisfied. Thus we get, whenever $k\in\N$, that
\begin{align}
&H_{\alpha}(z^k)-H_{\alpha}(z^{k+1})\ge C_1\norm{z^{k+1}-z^k}^2,\label{global ine 1}\\
&\norm{\nabla H_{\alpha}(z^{k})}\le C_2\norm{z^{k+1}-z^k},\label{global ine 2}
\end{align}
where $\alpha, C_1,C_2$ are defined in \eqref{C1C2}. It follows from \eqref{global ine 1} that the sequence $\set{H_\alpha(z^k)}$ is nonincreasing. Combining this with $\inf_{k\in\N} f(x^k)>-\infty$ and the inequality $H_\alpha(x,y)\ge f(x)$ for all $x,y\in\R^n$, we deduce that $\set{H_\alpha(z^k)}$ converges to some finite value. As a consequence, \eqref{global ine 1} tells us that
\begin{align*}
 C_1\sum_{k=1}^\infty \norm{z^{k+1}-z^k}^2\le \sum_{k=1}^\infty (H_\alpha(z^k)-H_\alpha(z^{k+1}))<\infty,
\end{align*}
which yields $\norm{z^{k+1}-z^k}\rightarrow 0$ and hence $\norm{x^{k+1}-x^k}\rightarrow 0$ as $k\rightarrow\infty.$ It follows from $x^{k+1}=x^k+\beta_k(x^k-x^{k-1})-\tau g^k$ and $0\le  \beta_k\le \bar \beta<\infty$ in Algorithm~\ref{IGDm} that
\begin{align*}
g^k=\frac{1}{\tau}\brac{x^k-x^{k+1}+\beta_k(x^k-x^{k-1})}\rightarrow 0\text{ as }k\rightarrow\infty.
\end{align*}
Combining $\norm{\nabla f(x^k_{\rm ex})-g^k}\le \nu\norm{g^k}$ for all $k\in\N$ with $g^k\rightarrow0$ gives us ${\nabla f(x^k_{\rm ex})}\rightarrow0$ as $k\rightarrow\infty.$ Since $\sup_{k\in\N}{|\beta_k-\gamma_k|}<\infty$ for $\beta_k\in [0,\bar \beta],$ we deduce that $\sup_{k\in\N}{\gamma_k} <\infty.$ Together with update \eqref{step xex} in Algorithm~\ref{IGDm}, the latter tells us that 
\begin{align*}
\norm{\nabla f(x^k_\rex)-\nabla f(x^k)}\le L \norm{x^k_{\rm ex}-x^k}=L\gamma_k\norm{x^{k}-x^{k-1}}\rightarrow0,
\end{align*}
which implies in turn that $\nabla f(x^k)\rightarrow0$ as $k\rightarrow\infty$ and this justifies (i) .

To verify assertions (ii) and (iii) simultaneously for the accumulation point $\bar x$ of $\set{x^k}$, we first observe that the stationarity of $\bar x$ for $f$ follows from Proposition~\ref{convergence xk+1 xk}. Further, it is easy to see that \eqref{global ine 1} and \eqref{global ine 2} correspond to the conditions in \eqref{two conditions} of Proposition~\ref{general rate} for the sequence $\set{z^k}\subset\R^{2n}$ and for the smooth function $H_\alpha(x,y)$. Observe also that $\nabla f(x^k_\rex)\ne 0$ and $\norm{g^k-\nabla f(x^k_\rex)}\le \nu \norm{g^k}$ yield $z^{k+1}\ne z^k$ for all $k\in\N$. Indeed, we deduce otherwise from the construction $z^k=(x^{k},x^{k-1})$ that $x^{k+1}=x^k=x^{k-1}$ for some $k\in\N$, which implies by \eqref{step xin} and \eqref{step xex} that $x^k_{\rm ex}=x^k_{\rm in}=x^k=x^{k+1}$. Then the update in \eqref{iteration update} tells us that $g^k=0$, and hence $\nabla f(x^k_{\rm ex})=0$ by $\norm{g^k-\nabla f(x^k_\rex)}\le \nu \norm{g^k}$, a contradiction.

By Proposition~\ref{general convergence under KL}, $\set{z^k}$ converges to $(\bar x,\bar x),$ which readily implies that $\set{x^k}$ converges to $\bar x$. Furthermore, it follows from Proposition~\ref{KL Lya global}(ii) that the imposed PLK property of $f$ with $\psi(t)=Mt^{1-q},\;M>0$, and $q\in [0,1)$ ensures that $H_\alpha$ satisfies the PLK property at $(\bar x,\bar x)$ with the desingularizing function $\varphi(t)=\overline Mt^{1-\theta}$, where $\theta=\max \set{q,\frac{1}{2}}\in [1/2,1),\overline{M}>0$. Thus we deduce from Proposition~\ref{general rate} that:\\ 
$\bullet$ For $q\in [0,1/2],$ we get $\theta=1/2$, which implies that $\set{z^k}$ converges linearly to $(\bar x,\bar x)$. This yields the linear convergence of the original sequence $\set{x^k}$ to $\bar x$.\\
$\bullet$  For $q\in (1/2,1)$, we get $\theta=q\in (1/2,1)$, which yields $\norm{z^k-(\bar x,\bar x)}=\mathcal{O}(k^{-\frac{1-q}{2q-1}})$ and therefore ensures that $\norm{x^k-\bar x}=\mathcal{O}(k^{-\frac{1-q}{2q-1}})$ as $k\rightarrow\infty.$

Now we use Proposition~\ref{convergence rate deduce} to verify convergence rates of $\set{\nabla f(x^k)}$ and $\set{f(x^k)}$. Since the gradient of $f$ is Lipschitz continuous on $\R^n$, the same property holds for $H_\alpha$ on $\R^n\times \R^n$, and hence condition \eqref{main condition rate} follows from \eqref{global ine 1} and \eqref{global ine 2}. This gives us the following:\\
$\bullet$ For $q\in[0,1/2],$ we get $\theta=1/2$ ensuring that $\{H_\alpha(z^k)\}$ converges linearly to $H_\alpha(\bar z)$ and that $\{\norm{\nabla H_\alpha(z^k)}\}$ converges linearly to $0$ as $k\to\infty$.\\
$\bullet$  For $q\in (1/2,1)$, we have that $H_\alpha(z^k)-H_\alpha(\tilde z)=\mathcal{O}(k^{-\frac{2-2q}{2q-1}})$ and $\norm{\nabla H_\alpha(z^k)}=\mathcal{O}(k^{-\frac{1-q}{2q-1}})$ as $k\to\infty$.

Finally, the convergence rates of the sequences $\{\norm{\nabla f(x^k)}\}$ and $\set{f(x^k)}$ are deduced from that for $\set{\norm{\nabla H_\alpha (z^k)}},\set{H_\alpha (z^k)}$ and the estimates
\begin{align*}
\norm{\nabla H_\alpha(x,y)}&=\sqrt{\norm{\nabla f(x)+2\alpha (x-y)}^2+4\alpha^2\norm{x-y}^2}\ge \norm{\nabla f(x)},\\ H_\alpha(x,y)&=f(x)+\alpha\norm{x-y}^2\ge f(x)\;\text{ for all }\;x,y\in\R^n.
\end{align*}
Therefore, we complete the proof of the theorem.
\end{proof}

The next theorem establishes {\em local convergence} properties of Algorithm~\ref{IGDm}. As follows from \cite[Theorem~7]{bento25}, the case of $q\in(0,1/2)$ in the exponential PLK condition for functions $f$ at their local minimizers is {\em inconsistent} with the local Lipschitz continuity of gradients, and hence should be excluded. Having this in mind leads us to the following local convergence analysis. Observe that, in contrast to the global convergence results in Theorem~\ref{global IGDm}, we now address {\em local minimizers} of $f$, not merely stationary points, and the presented proof is essentially more involved in comparison with the case of global convergence. 

\begin{Theorem}\label{local con IGDm}
Let the objective function $f:X\rightarrow\R$ in \eqref{main problem} satisfy the basic PLK condition at its local minimizer $\bar x$, and let $\set{x^k}$ be the sequence of iterates generated by Algorithm~{\rm\ref{IGDm}} with $X=\B(\bar x,r)$ for some $r>0$ and $\bar \beta\in [0,\sqrt{1-\nu})$. Then there exist $\xi \in(0,r]$ and $T>0$ such that for any initial point $x^0\in \mathbb{B}(\bar x, \xi)$ and $\tau\in (0, T]$, we have the convergence properties$:$
\begin{enumerate}
\item[\bf(i)] The sequences $\set{x^k},\set{x^k_{\rm in}},\set{x^k_{\rm ex}}$ stay in $\mathbb{B}(\bar x,r)$ and converge to some local minimizer $\tilde x\in\B(\bar x,r)$ of $f$ with $f(\tilde x)=f(\bar x)$.

\item[\bf(ii)] If the exponential PLK condition is satisfied with the desingularizing function $\psi(t)= Mt^{1-q}$ for some $M>0$ and $q\in(0,1)$, then the following convergence rates are guaranteed as $k\to\infty:$
\end{enumerate}\vspace*{-0.1in}

 $\bullet$ For $q=1/2$, the sequences $\set{x^k},\set{\nabla f(x^k)},\set{f(x^k)}$ converge linearly to $\tilde x,0,f(\tilde x),$ respectively.

 $\bullet$ For $q\in(1/2,1)$, we have the asymptotic estimates
 $$
 \norm{x^k-\tilde x }=\mathcal{O}\brac{k^{-\frac{1-q}{2q-1}}},\quad\norm{\nabla f(x^k)}=\mathcal{O}\brac{k^{-\frac{1-q}{2q-1}}},\quad f(x^k)-f(\tilde x)=\mathcal{O}\brac{k^{-\frac{2-2q}{2q-1}}}.
 $$
 \end{Theorem}
\begin{proof} To verify (i), deduce from $\bar \beta \in[0,\sqrt{1-\nu})$ that there exists $T>0$ such that
\begin{align}\label{LT<1}
T< \frac{(1-\bar \beta)(1-\nu)}{4L\sqrt{2}(1+\bar \delta)},\quad(LT+1)\bar\beta^2+2LT\bar \delta<1-\nu,\quad LT\le 1-\nu.
\end{align}
Define further $\varepsilon:=\min\set{\frac{1-\nu}{4T},\frac{1}{4}}\in \big(0,\frac{1}{4}\big]$ and deduce from Lemma~\ref{isolation lemma} due to the imposed assumptions that there exists a neighborhood $U\subset\B(\bar x,r)$ of $\bar x$ such that $\bar x$ minimizes $f$ on $U$ and that $f(\bar x)$ is the only critical value within this set. Remembering that $f$ satisfies the basic PLK condition at $\bar x$ tells us by Proposition~\ref{isolated} that the Lyapunov function $\set{H_\alpha}_{\alpha\ge \varepsilon}$ enjoys the same PLK property at $\bar z:=(\bar x,\bar x)$ for some $\eta\in (0,\infty]$, some neighborhood $Z$ of $\bar z:=(\bar x,\bar x)$, and a concave desingularizing function $\varphi:[0,\eta)\rightarrow[0,\infty)$ with $\varphi(0)=0$ which is ${\cal C}^1$-smooth on $(0,\eta)$. This means that for any $\alpha\ge \varepsilon$ and any $z\in Z$ with $H_\alpha(\bar z)< H_\alpha( z)<H_\alpha(\bar z)+\eta$, we get the inequality
\begin{align}\label{KL R2n in proof}
\varphi'\big(H_\alpha(z )-H_\alpha(\bar z )\big)\norm{\nabla H_\alpha(z)}\ge 1.
\end{align}
The continuity of $f$ allows us to find $\rho\in(0,r)$ such that 
\begin{align}\label{choice rho}
f(x)+\frac{L^2T\rho^2}{2(1-\bar \beta)^2(1-\nu)^2}< f(\bar x)+\eta\text{ for all }x \in \mathbb{B}(\bar x,\rho).
\end{align}
Invoking further the continuity of $\varphi$ at $0$ with $\varphi(0)=0$ gives us $\xi\in (0,\rho]$ for which
\begin{align}\label{defi xi 1}
\sqrt{2}\norm{x-\bar x}+\frac{4\sqrt{2}\max\set{\nu+1,\bar\beta(\nu+2)+\bar\delta}+8}{1-\nu-(L  T+1)\bar\beta^2-2LT \bar \delta}\varphi(f(x)-f(\bar x))&<\frac{\rho}{4}\;\text{ for all }\;x\in\mathbb{B}(\bar x,\xi).
\end{align}
Given the sequence of iterates $\set{x^k}$ generated by Algorithm~\ref{IGDm} with $x^0\in \mathbb{B}(\bar x, \xi)$ and $\tau\in (0, T]$, take the constant $\alpha,C_1,C_2$ from \eqref{C1C2} and deduce that
\begin{align}\label{sele alpha}
\alpha=\frac{(L\tau+1)\bar\beta^2+1-\nu}{4\tau}\ge \frac{1-\nu}{4T}\ge \varepsilon.
\end{align}
Since $\tau\le T$, we get from the selection of $T$ in \eqref{LT<1} the inequalities
\begin{align}\label{ineq 1-nu}
(L\tau+1)\bar \beta^2+2L\tau \bar \delta\le (LT+1)\bar \beta^2+2LT \bar \delta <1-\nu.
\end{align}
which bring us to the relationships  
\begin{align}\label{ineq 1/2}
\alpha=\frac{(L\tau+1)\bar\beta^2+1-\nu}{4\tau}\le \frac{2(1-\nu)-2L\tau \bar \delta}{4\tau}=\frac{1-\nu-L\tau \bar \delta}{2\tau}\le \frac{1}{2\tau}.
\end{align} 
It follows from $L\tau \le LT\le 1$ by \eqref{LT<1} and the estimates in \eqref{ineq 1/2}, \eqref{ineq 1-nu} that
\begin{align*}
\frac{C_2}{C_1}&= \frac{\sqrt{2}\max\set{\frac{\nu+1}{\tau },\frac{\bar\beta(L\tau+\nu+1)+L\tau \bar \delta}{\tau }}+4\alpha}{\frac{1-\nu-(L \tau+1)\bar \beta^2-2L\tau \bar \delta}{4\tau}}\\
&= \frac{4\sqrt{2}\max\set{{\nu+1}, \bar\beta(L\tau+\nu+1)+L\tau \bar \delta}+16\alpha\tau}{1-\nu-(L \tau+1)\bar\beta^2-2L\tau \bar \delta}\\
&\le \frac{4\sqrt{2}\max\set{\nu+1,\bar\beta(\nu+2)+\bar\delta}+8}{1-\nu-(L  T+1)\bar\beta^2-2LT \bar \delta}.
\end{align*}
This implies by \eqref{defi xi 1} that 
\begin{align}\label{defi xi 2}
\sqrt{2}\norm{x-\bar x}+\frac{C_2}{C_1}\varphi(f(x)-f(\bar x))&<\frac{\rho}{4} \text{ for all }x\in\mathbb{B}(\bar x,\xi).
\end{align}
The rest of the proof of assertion (i) is split into two claims.

\begin{Claim} The sequences $\set{x^k},\set{x^k_{\rm in}}$, and $\set{x^k_{\rm ex}}$ stay inside $\mathbb{B}(\bar x,\rho)$ for all $k\in\N$.
\end{Claim}
\noindent Since $x^0=x^1\in \mathbb{B}(\bar x,\xi),$ the updates in \eqref{step xin}, \eqref{step xex} of Algorithm~\ref{IGDm} tell us that $x^1_{\rm in}=x^1_{\rm ex}=x^1\in \mathbb{B}(\bar x,\xi)\subset \mathbb{B}(\bar x,\rho).$ We proceed now by induction to show that $x^k,x^k_{\rm in},x^k_{\rm ex}\in \mathbb{B}(\bar x,\rho)$ for all $k\in\N.$ Assuming that $x^k,x^k_\rin,x^k_{\rm ex}\in \mathbb{B}(\bar x,\rho)$ for $k=1,\ldots,K$, let us verify that $x^{K+1},x_\rin^{K+1},x^{K+1}_\rex\in \mathbb{B}(\bar x,\rho).$ It follows from \eqref{ineq 1-nu} and $L\tau\le LT<1-\nu$ by \eqref{LT<1} that assumption \eqref{max Lt<1-nu} in Proposition~\ref{prop AIGD tech} is satisfied. Therefore, we get the estimates
\begin{align}\label{descent con for rate 1}
&C_1 \norm{z^{k+1}-z^k}^2\le H_{\alpha}(z^k)-H_{\alpha}(z^{k+1}),\\
&\norm{\nabla H_{\alpha}(z^{k})}\le C_2\norm{z^{k+1}-z^k}\text{ for all }k=1,\ldots,K-1,\label{descent con for rate 2}
\end{align}
where $z^k=(x^k,x^{k-1})$ as before. Fix any such $k,$ we deduce from the above inequalities that
\begin{align}\label{primary descent condi}
\norm{z^{k+1}-z^k}\cdot\norm{\nabla H_\alpha(z^k)}\le C_2  \norm{z^{k+1}-z^k}^2\le \frac{C_2}{C_1}(H_\alpha(z^k)-H_\alpha(z^{k+1})).
\end{align}
To use inequality \eqref{KL R2n in proof} with $z:=z^k$, it is sufficient to check that $z^k\in Z$ and $ H_\alpha(\bar z)<H_\alpha(z^k)<H_\alpha(\bar z)+\eta$. Note first that the former condition is satisfied due to
\begin{align*}
z^k=(x^k,x^{k-1})\in \mathbb{B}(\bar x,\rho)\times \mathbb{B}(\bar x,\rho)\subset Z.
\end{align*}
 To verify the latter condition, deduce from Proposition~\ref{proposition distance} that $\norm{x^k-x^{k-1}}\le \frac{L\tau\rho}{(1-\bar \beta)(1-\nu)}$. Combining this with the construction of $H_\alpha$, the inequality $\alpha\le \frac{1}{2\tau}$ from \eqref{ineq 1/2}, the algorithm inclusions $x^k,x^{k-1}\in\mathbb{B}(\bar x,\rho)$, and the choice of $\rho$ in \eqref{choice rho} tells us that
\begin{align*}
H_\alpha(z^k)&=f(x^k)+\alpha\norm{x^k-x^{k-1}}^2\\
&\le f(x^k)+\frac{1}{2\tau}\cdot\frac{L^2\tau^2\rho^2}{(1-\bar \beta)^2(1-\nu)^2}\\
&= f(x^k)+\frac{L^2T\rho^2}{2(1-\bar \beta)^2(1-\nu)^2}< f(\bar x)+\eta=H_\alpha(\bar z)+\eta.
\end{align*}
Observe that the inequality $H_\alpha(\bar z)<H_\alpha(z^k)$ is also satisfied. Indeed, its failure means that 
$$
f(x^k)+\alpha\norm{x^{k}-x^{k-1}}^2\le H_\alpha(\bar z)=f(\bar x).
$$
By taking into account that $f(x^k)\ge f(\bar x)$ by $x^k\in\mathbb{B}(\bar x,\rho)$, the latter implies that
\begin{align*}
f(x^k)=f(\bar x) \text{ and }x^{k}=x^{k-1},
\end{align*}
which  means that $\nabla f(x^k)=0$ since $\bar x$ is a local minimizer of $f$ within $\B(\bar x,\rho)$. The equalities above also imply that
$x^k_{\rm ex}=x^k$ by \eqref{step xex}, and so $\nabla f(x^k_{\rm ex})=0,$ which contradicts Assumption~\ref{assu} and thus justifies \eqref{KL R2n in proof} with $z=z^k$. Multiplying both sides of \eqref{primary descent condi} by $\varphi'(H_\alpha(z^k)-H_\alpha(\bar z))$ and using \eqref{KL R2n in proof} give us
\begin{align}\label{ineq alter}
\norm{z^{k+1}-z^k}& \le \varphi'(H_\alpha(z^k)-H_\alpha(\bar z))\norm{\nabla H_\alpha(z^k)}\cdot\norm{z^{k+1}-z^k}\\
&\le \frac{C_2}{C_1}\varphi'(H_\alpha(z^k)-H_\alpha(\bar z))\brac{ H_{\alpha}(z^k)-H_{\alpha}(z^{k+1})}\nonumber\\
&\le\frac{C_2}{C_1}\brac{\varphi(H_\alpha(z^k)-H_\alpha(\bar z))-\varphi(H_\alpha(z^{k+1})-H_\alpha(\bar z))} \text{ for all }k=1,\ldots,K-1,
\end{align}
where the inequality  in \eqref{chain ineq} follows from the concavity of $\varphi.$
Combining further the obtained estimate with $\tau\le T<\frac{(1-\bar \beta)(1-\nu)}{4L\sqrt{2}}$ from \eqref{LT<1} and estimate \eqref{3.3(ii)} from Proposition~\ref{proposition distance},  we have
\begin{align}
\norm{z^{K+1}-\bar z}&\le\norm{z^1-\bar z}+\norm{z^{K+1}-z^K}+\sum_{k=1}^{K-1}\norm{z^{k+1}-z^k}\nonumber\\
&\le \norm{z^1-\bar z}+\norm{z^{K+1}-z^K}+\frac{C_2}{C_1}\sum_{k=1}^{K-1}\brac{\varphi(H_\alpha(z^k)-H_\alpha(\bar z))-\varphi(H_\alpha(z^{k+1})-H_\alpha(\bar z))}\nonumber\\
&=\sqrt{2}\norm{x^1-\bar x}+\frac{\sqrt{2}L\rho\tau}{(1-\bar \beta)(1-\nu)}+\frac{C_2}{C_1}\brac{\varphi(H_\alpha(z^1)-H_\alpha(\bar z))-\varphi(H_\alpha(z^{K})-H_\alpha(\bar z))}\nonumber\\
&\le \sqrt{2}\norm{x^1-\bar x}+\frac{\rho}{4}+\frac{C_2}{C_1}\varphi(f(x^1)-f(\bar x))<\rho/2,\label{chain ineq}
\end{align}
where the last inequality is a consequence of \eqref{defi xi 2} and $x^1\in\B(\bar x,\xi)$. 
Therefore, 
\begin{align*}
\norm{x^{K+1}-\bar x}\le \norm{z^{K+1}-\bar z}\le \rho/2,
\end{align*}
which means that $x^{K+1}\in\mathbb{B}(\bar x,\rho/2)$. Using the updates in \eqref{step xin} and \eqref{step xex} at the $(K+1)^{\rm th}$ iteration with taking into account \eqref{3.3(ii)} from Proposition~\ref{proposition distance}, $\tau\le T<\frac{(1-\bar \beta)(1-\nu)}{4L\sqrt{2}(1+\bar\delta)}$ from \eqref{LT<1}, and  
\begin{align*}
0\le \beta_{K+1}<1,\quad 0\le \delta_{K+1}\le \delta_{K+1}-\beta_{K+1}+\beta_{K+1}\le \bar \delta +1,
\end{align*}
we arrive at the estimates
\begin{align*}
\norm{x^{K+1}_\rin-\bar x}&\le \norm{x^{K+1}-\bar x}+\norm{x^{K+1}-x^K}\le \rho/2+ \frac{2L\tau\rho}{(1-\bar \beta)(1-\nu)}<\rho,\\
\norm{x^{K+1}_\rex-\bar x}&\le \norm{x^{K+1}-\bar x}+(1+\bar \delta)\norm{x^{K+1}-x^K}\le \rho/2+ \frac{2L\tau\rho(1+\bar\delta)}{(1-\bar \beta)(1-\nu)}<\rho.
\end{align*}
Thus $x^{K+1}_\rin,x^{K+1}_\rex\in\mathbb{B}(\bar x,\rho)$, which completes the claimed assertion by induction. 
\medskip

\noindent\textbf{Claim~2.} \textit{The sequences $\set{x^k},\set{x^k_{\rm in}},\set{x^k_{\rm ex}}$ converge to a local minimizer $\tilde x$ of $f$ with $f(\tilde x)=f(\bar x)$.}
\medskip

\noindent Indeed, it follows from \eqref{ineq alter} and $x^k,x^k_\rin,x^k_{\rm ex}\in\mathbb{B}(\bar x,\rho)$ from Claim 1 that 
\begin{align*}
\norm{z^{k+1}-z^k}&\le \frac{C_2}{C_1}\brac{\varphi(H_\alpha(z^k)-H_\alpha(\bar z))-\varphi(H_\alpha(z^{k+1})-H_\alpha(\bar z))} \text{ for all }k\in\N.
\end{align*}
This readily implies that 
\begin{align*}
\sum_{k=1}^\infty \norm{z^{k+1}-z^k}\le \frac{C_2}{C_1}\varphi(H_\alpha(z^1)-H_\alpha(\bar z))<\infty.
\end{align*}
Therefore, $\set{z^k}$ is convergent, and thus $\set{x^k}$ converges to some $\tilde x\in\overline{\mathbb{B}}(\bar x,\rho)\subset U$. Using  Proposition~\ref{convergence xk+1 xk}, we justify the convergence of $\set{x^k_\rex},\set{x^k_\rin}$ to $\tilde x$ and that $\tilde x$ is a stationary point of $f.$ Since $f(\bar x)$ is the only critical value of $f$ within $U,$ it gives us $f(\tilde x)=f(\bar x)$ and shows that $\tilde x$ is a local minimizer of $f$.\vspace*{0.05in} 

To verify (ii), we employ Proposition~\ref{general rate}. Letting $\tilde z:=(\tilde x,\tilde x)$ and using \eqref{descent con for rate 1}, \eqref{descent con for rate 2} from the proof of Claim~1 in (i) together with $\set{x^k},\set{x^k_{\rm in}},\set{x^k_{\rm ex}}\subset\mathbb{B}(\bar x,\rho)$ give us the estimate
\begin{align}\label{ineq rate grad}
\begin{array}{ll}
C_1 \norm{z^{k+1}-z^k}^2\le H_{\alpha}(z^k)-H_{\alpha}(z^{k+1}),\\
\norm{\nabla H_{\alpha}(z^{k})}\le C_2\norm{z^{k+1}-z^k},\quad k\in\N.
\end{array}
\end{align}
Observe from $\nabla f(x^k_\rex)\ne 0$ and $\norm{g^k-\nabla f(x^k_\rex)}\le \nu \norm{g^k}$ that $z^{k+1}\ne z^k$ for all $k\in\N.$ Indeed, the contrary implies by $z^k=(x^{k},x^{k-1})$ that $x^{k+1}=x^k=x^{k-1}$ for some $k\in\N$. Using the updates in \eqref{step xin}, \eqref{step xex} of Algorithm~\ref{IGDm} ensures that $x^k_{\rm ex}=x^k_{\rm in}=x^k=x^{k+1}$. Then $g^k=0$ by \eqref{iteration update}, and hence $\nabla f(x^k_{\rm ex})=0$ by $\norm{g^k-\nabla f(x^k_\rex)}\le \nu \norm{g^k}$. This contradicts the assumption $\nabla f(x^k_{\rm ex})\ne 0$ for all $k\in\N$.

Recalling that $f$ satisfies the PLK condition at $\bar x$ with the desingularizing function $\psi(t)=Mt^{1-q}$, we deduce from Proposition~\ref{isolated}(ii) that $H_\alpha$ satisfies the PLK condition at $(\bar x,\bar x)$ with the neighborhood $Z$ taking from \eqref{KL R2n in proof} and the desingularizing function $\varphi(t)=\overline Mt^{1-\theta}$, where $\theta=\max \set{q,\frac{1}{2}}\in (0,1),\overline{M}>0.$ Taking into account Remark~\ref{algebraic}, this PLK property also holds at $\tilde z\in\overline{\mathbb{B}}(\bar x,\rho)\times \overline{\mathbb{B}}(\bar x,\rho)\subset Z$ due to $H_\alpha(\tilde z)=f(\tilde x)=f(\bar x)=H_\alpha(\bar z).$  Then Proposition~\ref{general rate} applied to the function $H_\alpha$ and sequence $\set{z^k}\subset\R^{n}\times \R^n$, with taking into account Remark~\ref{rmk general rate}, brings us to the following:
\begin{itemize}
\item For $q\in[0,1/2],$ we get $\theta=1/2$, which ensures that $\set{z^k}$ converges linearly to $(\tilde x,\tilde x)$. This implies that $\set{x^k}$ converges linearly to $\tilde x.$

\item For $q\in (1/2,1),$ we get $\theta=q\in (1/2,1)$, which tells us that $\norm{z^k-(\tilde x,\tilde x)}=\mathcal{O}(k^{-\frac{1-q}{2q-1}})$ and thus yields the convergence rate $\norm{x^k-\tilde x}=\mathcal{O}(k^{-\frac{1-q}{2q-1}})$ as $k\rightarrow\infty$.
\end{itemize}

Finally, we employ Proposition~\ref{convergence rate deduce} to verify the convergence rates of $\set{\nabla f(x^k)}$ and $\set{f(x^k)}.$ Since $f$ is of class ${\cal C}^{1,1}$ on $\B(\bar x,r)$ and since $H_\alpha$ is of class ${\cal C}^{1,1}$ on $\B(\bar x,r)\times \B(\bar x,r)\supset \set{z^k}$, condition \eqref{main condition rate} follows from \eqref{ineq rate grad}. This clearly leads us to the conclusions:
\begin{itemize}
\item For $q\in[0,1/2),$ we get $\theta=1/2$. Therefore, $\{H_\alpha(z^k)\}$ converges linearly to $H_\alpha(\bar z)$ and $\{\norm{\nabla H_\alpha(z^k)}\}$ converges linearly to $0$ as $k\to\infty$.

\item For $q\in (1/2,1),$ we have $H_\alpha(z^k)-H_\alpha(\tilde z)=\mathcal{O}(k^{-\frac{2-2q}{2q-1}})$ and $\norm{\nabla H_\alpha(z^k)}=\mathcal{O}(k^{-\frac{1-q}{2q-1}})$ as $k\to\infty$.
\end{itemize}
The claimed convergence rates of the sequences $\{\norm{\nabla f(x^k)}\}$ and $\set{f(x^k)}$ are now deduced from those for $\set{\norm{\nabla H_\alpha (z^k)}},\set{H_\alpha (z^k)}$ and the estimates
\begin{align*}
\norm{\nabla H_\alpha(x,y)}&=\sqrt{\norm{\nabla f(x)+2\alpha (x-y)}^2+4\alpha^2\norm{x-y}^2}\ge \norm{\nabla f(x)},\\
H_\alpha(x,y)&=f(x)+\alpha\norm{x-y}^2\ge f(x)\;\text{ for all }\;x,y\in\R^n.
\end{align*}
This completes the proof of the theorem.
\end{proof}\vspace*{-0.2in}

\section{Convergence of First-Order Methods with Momentum}\label{sec: illustrate}

In this section, we design and justify momentum versions of the three first-order methods, which are well recognized in optimization and its applications; namely, the extragradient and inexact proximal point methods, and sharpness-aware minimization. Our goal is to demonstrate that the convergence analysis of IGDm developed above induces convergence results for the corresponding momentum algorithms.\vspace*{-0.1in}

\subsection{Extragradient Method with Momentum}\label{subsection EGm}

The extragradient algorithm with momentum (EGm) for solving \eqref{main problem} is design as follows.\vspace*{0.05in}

\begin{longfbox}
\begin{Algorithm}[EGm]\hlabel{EG algorithm}\quad
\begin{enumerate}[-]
\item \textbf{Optimization problem}:
\begin{align*}
\eqref{main problem}\text{ with }\begin{cases}
X=\R^n\;\text{ for global minima,}\\
X=\B(\bar x,r)\;\text{ for a local minimizer }\;\bar x\;\text{ with some  }\;r>0.
\end{cases}
\end{align*}
\item  \textbf{Initialization:} Choose $x^0=x^1\in X,\set{\beta_k},\set{\gamma_k}\subset[0,\infty),\tau_1,\tau_2>0$.
\item \textbf{Parameter conditions:} $\bar \beta:=\sup\beta_k<1,\bar\delta:=\sup|\beta_k-\gamma_k|<\infty.$
\item \textbf{Iteration:} $(k\ge 1)$ Update:
\begin{align}
&x^k_{\rm in}:=x^k+\beta_k(x^k-x^{k-1}),\label{xin EG}\\
&x^k_{\rm ex}:=x^k+\gamma_k(x^k-x^{k-1}),\label{xex EG}\\
&x^{k+1}:=x^k_{\rm in}-\tau_1 \nabla f(x^k_{\rm ex} -\tau_2 \nabla f(x^k_{\rm ex}))\label{update EG}.
\end{align}
\end{enumerate}
\end{Algorithm}    
\end{longfbox}\vspace*{0.1in}

Let us show that whenever $\tau_2 \le \frac{\nu}{L(\nu+1)}$ for $\nu \in(0,1)$, Algorithm~\ref{EG algorithm} is a special case of Algorithm~\ref{IGDm} with $\tau:=\tau_1$. It is sufficient to verify that for each $k\in\N$, we have
\begin{align*}
\norm{g^k-\nabla f(x^k_\rex)}\le \nu \norm{g^k},\;\text{ where }\;g^k:= \nabla f(x^k_{\rm ex} -\tau_2 \nabla f(x^k_{\rm ex})).
\end{align*}
Pick $k\in\N$ and deduce from the Lipschitz continuity of $\nabla f$ on $X$ and the condition $L\tau_2\le \frac{\nu}{\nu+1}$ that
\begin{align}\label{com 1}
\norm{g^k-\nabla f(x^k_\rex)}&= \norm{\nabla f(x^k_{\rm ex} -\tau_2 \nabla f(x^k_{\rm ex}))-\nabla f(x^k_\rex)}\nonumber\\
&\le L\tau_2\norm{\nabla f(x^k_\rex)}\le \frac{\nu}{\nu+1}\norm{\nabla f(x^k_\rex)}\\
&\le \frac{\nu}{\nu+1}\norm{\nabla f(x^k_\rex)-g^k}+\frac{\nu}{\nu+1}\norm{g^k}.
\end{align}
This brings us to the estimate 
\begin{align*}
\frac{1}{\nu+1}\norm{\nabla f(x^k_\rex)-g^k}\le \frac{\nu}{\nu+1}\norm{g^k},\;\mbox{ i.e., }\;\norm{\nabla f(x^k_\rex)-g^k}\le {\nu}\norm{g^k}.
\end{align*}
Therefore, we arrive at the following global and local convergence results for EGm that are immediate consequences of Theorem~\ref{global IGDm} and Theorem~\ref{local con IGDm}, respectively.

\begin{Theorem}\label{glo EG}
Let $\set{x^k}$ be the sequence of iterates generated by Algorithm~{\rm\ref{EG algorithm}} with $X=\R^n,\;\tau_2\le \frac{\nu}{L(\nu+1)}$ for some $\nu\in(0,1)$, and $\max\set{L\tau_1,2L\tau_1\bar \delta+(L\tau_1+1)\bar\beta^2}<1-\nu$. If $\inf_{k\in\N} f(x^k)>-\infty,$ then all the assertions of Theorem~{\rm\ref{global IGDm}} hold for $\set{x^k}$.
\end{Theorem}
\begin{Theorem}\label{lo EG}
Let $\set{x^k}$ be the sequence of iterates generated by Algorithm~{\rm\ref{EG algorithm}} with $X=\B(\bar x,r)$ for some local minimizer $\bar x$ of $f$ with $r>0$. Then given $\nu\in(0,1)$ and $\bar \beta\in [0,\sqrt{1-\nu})$, there exist numbers $\xi > 0$ and $T_1>0$ such that for any initial points $x^0\in \mathbb{B}(\bar x, \xi)$, $\tau_1\in (0, T]$, and $\tau_2\in \sbrac{0,\frac{\nu}{L(\nu+1)}}$, the sequences $\set{x^k},\set{x^k_\rex},\set{x^k_\rin}$ generated by Algorithm~{\rm\ref{EG algorithm}} have the convergence properties of Theorem~{\rm\ref{local con IGDm}}.
\end{Theorem}\vspace*{-0.15in}

\subsection{Sharpness-Aware Minimization with Momentum}\label{SAMm}
This subsection designs the algorithm of sharpness-aware minimization with momentum (SAMm).\vspace*{0.05in}

\begin{longfbox}
\begin{Algorithm}[SAMm]\label{SAM algorithm}\quad
\begin{enumerate}[-]
\item \textbf{Optimization problem}:
\begin{align*}
\eqref{main problem}\;\text{ with }\begin{cases}
X=\R^n\;\text{ for global minima},\\
X=\B(\bar x,r)\;\text{ for a local minimizer }\;\bar x\;\text{ with some  }\;r>0.
\end{cases}
\end{align*}
\item  \textbf{Initialization:} Choose $x^0=x^1\in X,\set{\beta_k},\set{\gamma_k}\subset[0,\infty),\tau_1,\tau_2>0$.
\item \textbf{Parameter conditions:} $\bar \beta:=\sup\beta_k<1,\bar\delta:=\sup|\beta_k-\gamma_k|<\infty.$
\item \textbf{Iteration:} $(k\ge 1)$ Update:
\begin{align*}
&x^k_{\rm in}:=x^k+\beta_k(x^k-x^{k-1}),\\
&x^k_{\rm ex}:=x^k+\gamma_k(x^k-x^{k-1}),\\
&x^{k+1}:=x^k_{\rm in}-\tau_1 \nabla f(x^k_{\rm ex} +\tau_2 \nabla f(x^k_{\rm ex})).
\end{align*}
\end{enumerate}
\end{Algorithm}    
\end{longfbox}\vspace*{0.1in}

\noindent Similarly to EGm in Subsection~\ref{subsection EGm}, observe that whenever $\tau_2\le \frac{\nu}{L(\nu+1)}$ for some $\nu\in (0,1)$, Algorithm~\ref{SAM algorithm} is a special case of Algorithm~\ref{IGDm} with $\tau:=\tau_1$. Thus we deduce the following convergence properties.

\begin{Theorem}\label{glo SAM}
Let $\set{x^k}$ be the sequence of iterates generated by Algorithm~{\rm\ref{SAM algorithm}} with $X=\R^n,\;\tau_2\le \frac{\nu}{L(\nu+1)}$, and $\max\set{L\tau_1,2L\tau_1\bar \delta+(L\tau_1+1)\bar\beta^2}<1-\nu$. If $\inf_{k\in\N} f(x^k)>-\infty,$ then all the assertions of Theorem~{\rm\ref{global IGDm}} hold for $\set{x^k}$.
\end{Theorem}
\begin{Theorem}\label{lo SAM}
Let $\set{x^k}$ be the sequence of iterates generated by Algorithm~{\rm\ref{SAM algorithm}} with $X=\B(\bar x,r)$ for some local minimizer $\bar x$ of $f$ and $r>0$. Then given $\nu\in(0,1)$ and $\bar \beta\in [0,\sqrt{1-\nu})$, there exist numbers $\xi > 0$ and $T_1>0$ such that for any initial points $x^0=x^1 \in \mathbb{B}(\bar x, \xi)$, $\tau_1\in (0, T]$, and $\tau_2\in \sbrac{0,\frac{\nu}{L(\nu+1)}}$, the sequences $\set{x^k},\;\set{x^k_\rex},\;\set{x^k\rin}$ generated by Algorithm~{\rm\ref{SAM algorithm}} satisfy the assertions in Theorem~{\rm\ref{local con IGDm}}.
\end{Theorem}\vspace*{-0.15in}

\subsection{Inexact Proximal Point Method with Momentum}\label{subsection IPPM}

This subsection addresses the following optimization problem:
\begin{align}\label{PPM problem}
{\rm minimize}\quad h(x)\;\text{ subject to }\;x\in\R^n,
\end{align}
where $h:\R^n\rightarrow\overline{\R}$ is an extended-real valued function. Although problems \eqref{main problem} and \eqref{PPM problem} are written in a similar unconstrained format, there exists a major difference between minimizing a smooth objective in \eqref{main problem} and a nonsmooth extended-real-valued objective in \eqref{PPM problem}. Since  \eqref{PPM problem} mandatorily contains the implicit constraint $x\in\dom h$, it provides a bridge to study explicitly constrained optimization problems. The {\em inexact proximal point method} to solve convex problems \eqref{PPM problem} was developed in \cite{Rockafellar1976}. The design and justification of the inexact proximal point algorithm in \cite{Rockafellar1976} was based on the classical tools of convex analysis and theory of monotone operators. 

Now we proceed with developing a nonconvex version of the inexact proximal point algorithm with momentum and establish its convergence properties based on the above results for IGDm and some fundamental constructions of variational analysis. First we recall the following notions.

\begin{Definition} {\rm 
For a proper extended-real-valued function $h:\R^n\rightarrow\overline{\R}$ and a parameter $\lambda>0$, the {\em Moreau envelope} $e_\lambda h$ and the associated {\em proximal mapping} $\prox_{\lambda h}$ are defined by
\begin{align}\label{defi of envelope}
e_\lambda h(x)&:=\inf \set{h(y)+\frac{1}{2\lambda}\norm{y-x}^2\;\Big|\;y\in\R^n},\\
\prox_{\lambda h}(x)&:={\rm argmin} \set{h(y)+\frac{1}{2\lambda}\norm{y-x}^2\;\Big|\;y\in\R^n}.\label{Moreau}
\end{align}
A function $h:\R^n\rightarrow \Bar\R$ is called {\em prox-bounded} if there is $\lambda>0$ such that $e_\lambda h(x)>-\infty$ for some $x\in\R^n$.}
\end{Definition}

The novel inexact proximal point method with momentum (IPPm) to solve \eqref{PPM problem} is defined as follows.\vspace*{0.05in}

\begin{longfbox}
\begin{Algorithm}[IPPm]\label{IPP algorithm}\quad
\begin{enumerate}[-]
\item \textbf{Optimization problem:}
\begin{align*}
\eqref{PPM problem}\;\text{ with }\;\begin{cases}
X=\R^n\;\text{ for global minimum, } \\
X=\B(\bar x,r)\;\text{ for a local minimizer }\;\bar x\;{ with some }\;r>0.
\end{cases}
\end{align*}
\item \textbf{Initialization:}  Choose $x^0=x^1\in X,\;\nu\in (0,1),\;\set{\beta_k},\;\set{\gamma_k}\subset[0,\infty),\;\lambda>0,\;\tau>0$.
\item \textbf{Parameter conditions:} $\bar{\beta}:=\sup \beta_k<1,$ $\bar \delta:=\sup|\beta_k-\gamma_k|<\infty.$
\item \textbf{Iteration: }$(k\ge 1)$ Update:
\begin{align}
&x^k_{\rm in}:=x^k+\beta_k(x^k-x^{k-1}),\nonumber\\
&x^k_{\rm ex}:=x^k+\gamma_k(x^k-x^{k-1}),\nonumber\\
&g^k:=\lambda^{-1}(x^k_{\rm ex}-p^k),\;\text{ where }\;\norm{p^k-\prox_{\lambda h}(x^k_{\rm ex})}\le \nu\norm{p^k-x^k_{\rm ex}}\label{PPM selection gk},\\
&x^{k+1}:=x^k_{\rm in}-\tau g^k.\nonumber
\end{align}
\end{enumerate}
\end{Algorithm}
\end{longfbox}\vspace*{0.1in}
\noindent

Our goal is to establish global and local convergence results for IPPm. To accomplish this, we consider two remarkable classes of generally nonconvex extended-real-valued functions enjoying amenable properties. which include the single-valuedness of the proximal mapping \eqref{Moreau}; see, e.g., \cite{RockafellarWets}.

\begin{Definition}\label{def:prox-reg}\rm Let  $h:\R^n\rightarrow \Bar\R$ be a proper l.s.c.\ function. Then we say that:
\begin{enumerate}
\item[\bf(i)] The function $h$ is {\em $\varrho$-weakly convex} on $\R^n$ with some $\varrho\ge 0$ if the quadratically shifted function $h(\cdot)+\dfrac{\varrho}{2}\norm{\cdot}^2$ is convex on $\R^n$.

\item[\bf(ii)] The function $h$ is {\em prox-regular} at $\bar x\in\dom h$ for $\bar v\in\partial h(\bar x)$ if there exist $\varepsilon>0$ and $\rho\ge 0$ such that 
\begin{align*}
 h(u)\ge h(x)+\dotproduct{v,u-x}-\frac{\rho}{2}\norm{u-x}^2
\end{align*}
whenever $x\in\mathbb{B}_\varepsilon(\bar x)$ and $(u,v)\in \gph \partial h\cap (\mathbb{B}_\varepsilon(\bar x)\times \mathbb{B}_\varepsilon(\bar v))$ with $h(x)<h(\bar x)+\varepsilon$.
\end{enumerate}
\end{Definition}

Now we present several propositions containing some properties of functions from both classes in Definition~\ref{def:prox-reg} that are used in what follows.

\begin{Proposition}\label{weakly and descent}
Let $h :\R^n\rightarrow\overline{\R}$ be a $\varrho$-weakly convex function and $\lambda\in(0,\varrho^{-1})$. Then we have:
\begin{enumerate} 
\item[\bf(i)]  The mapping $\prox_{\lambda h}$ is single-valued on $\R^n$.

\item[\bf(ii)] The Moreau envelope $e_{\lambda}g$ is $\mathcal{C}^1$-smooth and its gradient is expressed as
\begin{align}\label{prop 2.8 i}
\nabla e_{\lambda}h(x)=\lambda^{-1}(x-\prox_{\lambda h }(x)),\quad x\in\R^n,
\end{align}
while being globally Lipschitz continuous on $\R^n$ with modulus $\max \set{\lambda^{-1},\frac{\varrho}{1-\lambda\varrho}}$. 

\item[\bf(iii)] Every stationary point of $e_\lambda h$ is a stationary point of $h$.

\item[\bf(iv)] It holds that $\displaystyle{\inf_{x\in\R^n}} e_\lambda f(x)=\displaystyle{\inf_{x\in\R^n}} h(x)$.
\end{enumerate}
\end{Proposition}
\begin{proof} The results in (i) and (ii) follow from \cite[Lemma~2.5]{DavisDrusviatskiy2022}, while (iii) is established in \cite[Proposition~3.9]{KhanhMordukhovichPhatTran} in more generality. To verify (iv), we first deduced from \eqref{Moreau} that
\begin{align*}
e_\lambda h(x)\le h(x)\text{ for all }x\in\R^n,
\end{align*}
which implies that $\displaystyle{\inf_{x\in\R^n}} e_\lambda f(x)\le \displaystyle{\inf_{x\in\R^n}} h(x)$. The  opposite inequality follows from 
\begin{align*}
e_\lambda h(x) = h(\prox_{\lambda h}(x))+\frac{1}{2\lambda}\norm{\prox_{\lambda h}(x)-x}^2\ge h(\prox_{\lambda h}(x))\ge \inf_{x\in\R^n} h(x),
\end{align*}
and therefore the proof is complete.
\end{proof}

Next we recall some important properties of Moreau envelopes and proximal mappings associated with prox-regular functions; see  \cite[Proposition~13.37]{RockafellarWets}.

\begin{Proposition}\label{C11} Let $ h: \R^n \to \overline{\R}$ be prox-bounded on $\R^n$ and prox-regular at $\bar{x}$ for $\bar{v}$. Then	for all $\lambda>0$ sufficiently small, there is a neighborhood $U_\lambda$ of $\bar{x}+\lambda\bar{v}$ such that
\begin{enumerate}[\rm (i)]
\item[\bf(i)] ${\operatorname{Prox}}_{\lambda  h}$ is single-valued and Lipschitz continuous on $U_\lambda$ with ${\text{\rm Prox}}_{\lambda  h}(\bar{x}+\lambda\bar{v})=\bar{x}$. 

\item[\bf(ii)] $e_\lambda h$ is of class ${\cal C}^{1,1}$ on $U_\lambda$ with $\nabla e_\lambda h(\bar{x}+\lambda\bar{v})=\bar{v}$ and
\begin{equation}\label{GradEnvelope}
\nabla e_\lambda  h(x)=\lambda^{-1}(x-\operatorname{ Prox}_{\lambda h}(x))\;\mbox{ for all }\;x\in U_\lambda.
\end{equation}
\end{enumerate} 
\end{Proposition}

The preservation of local minimizers under the taking of Moreau envelopes, established in \cite[Theorem~3.2]{KhanhKhoaMordukhovichPhat}, plays a significant role in the subsequent convergence analysis of IPPm.

\begin{Proposition}\label{localMoreau}
Let $h:\R^n\rightarrow\overline{\R}$ be l.s.c.\ and prox-bounded on $\R^n$, and let $\bar x\in \R^n$. Then for any small $\lambda>0$, the following assertions are equivalent:
\begin{enumerate}[\bf (i)]
\item $\bar x$ is a local minimizer of $h$.

\item $\bar x$ is a local minimizer of $e_\lambda h$.
\end{enumerate}
\end{Proposition}

The next proposition shows that, for both classes of weakly convex and prox-regular functions $h:\mathbb{R}^n \rightarrow \overline{\mathbb{R}}$, Algorithm~\ref{IPP algorithm} applied to $h$ reduces to Algorithm~\ref{IGDm} applied to its Moreau envelope. 

\begin{Proposition}\label{PPM problem transformation} Let  $h: \mathbb{R}^n \rightarrow \overline{\mathbb{R}}$ be a proper extended-real-valued function. We have the assertions:

\begin{enumerate}[\bf (i)]
\item Assume that $h$ is  $\varrho$-weakly convex function, that $\lambda \in (0,\rho^{-1})$, and that $f: = e_\lambda h$. Then Algorithm~{\rm\ref{IPP algorithm}} applied to $h$ with such $\lambda$ agrees with Algorithm~{\rm\ref{IGDm}} applied to $f=e_\lambda h$ with $X = \mathbb{R}^n$.

\item Assume that $h$ is prox-bounded on $\R^n$ and prox-regular at some local minimizer $\bar x$ for $\bar v=0$. Then there exists $\bar \lambda>0$ such that for all $\lambda\in(0,\bar\lambda]$, we can select $r_\lambda>0$ so that Algorithm~{\rm\ref{IPP algorithm}} applied to $h$ reduces to Algorithm~{\rm\ref{IGDm}} for $f:=e_\lambda h$ on $X=\B(\bar x,r_\lambda)$.
\end{enumerate}
\end{Proposition}
\begin{proof}
To verify (i), deduce from Proposition~\ref{weakly and descent} that $f$ is $\C^{1,1}_L$ on $\R^n$ with $L=\max\set{\lambda^{-1},\frac{\varrho}{1-\lambda\varrho}}$. It follows from \eqref{prop 2.8 i} and the selection of $g^k$ in \eqref{PPM selection gk} that 
\begin{align}\label{esti weakly convex}
\begin{array}{ll}
&\norm{g^k-\nabla f(x^k_\rex)}=\norm{\lambda^{-1}(x^k_\rex-p^k)-\lambda^{-1}(x^k_\rex-\prox_{\lambda h}(x^k_\rex))}\\
&=\lambda^{-1}\norm{p^k-\prox_{\lambda h}(x^k_\rex)}\le \nu \lambda^{-1}\norm{p^k-x^k_\rex}=\nu \norm{g^k},
\end{array}
\end{align}
which clearly justifies assertion (i).

To verify (ii), we use Proposition~\ref{C11} and find $\bar \lambda>0$ such that for all $\lambda\in (0,\bar \lambda]$, there exists $r_\lambda>0$ ensuring that $f=e_\lambda h$ is of class $\C^{1,1}_L$ on $\B(\bar x,r_\lambda)$ with some $L>0$. Fix $\lambda$ as above and utilize \eqref{GradEnvelope} together with the update 
$g^k=\lambda^{-1}(x^k_\rex-p^k)$ to get \eqref{esti weakly convex}, which therefore completes the proof. 
\end{proof}

Now we are ready to establish global and local convergence results for Algorithm~\ref{IPP algorithm}. Similarly to Subsection~\ref{subsect convergence nanalys}, let us impose the natural requirement in what follows:

\begin{Assumption}\rm \label{assum IPPM}
$0\notin \partial h(x^k_\rex)$ for all $k\in\N.$
\end{Assumption}

\begin{Theorem}\label{theo IPPM global}
Let $h:\R^n\rightarrow\overline\R$ be a proper, l.s.c., and $\rho$-weakly convex function bounded from below, and let $\set{x^k}$ be the sequence of iterates generated by Algorithm~{\rm\ref{IPP algorithm}} with 
\begin{align}\label{lambda}
\lambda\in(0,\varrho^{-1})\;\mbox{ and }\;\max\set{L\tau,2L\tau \bar \delta+(L\tau+1)\bar\beta^2}<1-\nu,\;
\text{ where }\;L=\max\set{\lambda^{-1},\frac{\varrho}{1-\lambda \varrho}}.
\end{align}
Then the following assertions hold:
\begin{enumerate}[\bf(i)]
\item Every accumulation point $\bar x$ of $\set{x^k}$ is a stationary point of the function $h$.

\item If the Moreau envelope $e_\lambda h$ satisfies the basic PLK condition at an accumulation point $\bar x$ of $\set{x^k}$, then the sequence of iterates $\set{x^k}$ converges to $\bar x$ as $k\to\infty$.
 
\item If $e_\lambda h$ satisfies the exponential PLK condition with $\psi(t)=Mt^{1-q}$ for some $M>0$ and $q\in[0,1)$, then the convergence rates of $\set{x^k}$ and $\set{x^k-\prox_{\lambda h}(x^k)}$ are as follows:
\begin{itemize}

\item For $q\in[0,1/2]$, the sequences $\set{x^k}$ and $ \set{x^k-\prox_{\lambda h}(x^k)}$ converge linearly as $k\to\infty$ to $\bar x$ and $0$, respectively.

\item For $q\in(1/2,1)$, we have the convergence rates
$$
\norm{x^k-\bar x}=\mathcal{O}\brac{k^{-\frac{1-q}{2q-1}}}\;\mbox{ and }\;\norm{x^k-\prox_{\lambda h}(x^k)}=\mathcal{O}\brac{k^{-\frac{1-q}{2q-1}}}\;\mbox{ as }\;k\to\infty.
$$
\end{itemize}
\end{enumerate}\end{Theorem}
\begin{proof}
Select $\lambda$ and $L$ as in \eqref{lambda}. By Proposition~\ref{weakly and descent}, the Moreau envelope $f=e_\lambda h$ has a Lipschitz continuous gradient with constant $L$. Proposition~\ref{PPM problem transformation}(i) tells us that Algorithm~\ref{IPP algorithm} reduces to Algorithm~\ref{IGDm} for $f$. By Proposition~\ref{weakly and descent}(iv) and the boundedness from below of $f$, we get the condition $\inf_{k\in\N}f(x^k)>-\infty$. Let us further verify Assumption~\ref{assu} for $f$. Pick any $k\in\N$ and suppose that $\nabla f(x^k_\rex)=0$, which gives us by \eqref{prop 2.8 i} that $x^k_\rex=\prox_{\lambda h}(x^k_\rex)$. Using the subdifferential Fermat rule and the construction in  \eqref{Moreau} ensures that $0\in \partial h(x^k_\rex),$ which clearly violates Assumption~\ref{assum IPPM}. Now we can apply Theorem~\ref{global IGDm}(i) and conclude that every accumulation point of $\set{x^k}$ is a stationary point of $f=e_\lambda h$, which is an accumulation point of $h$ by Proposition~\ref{weakly and descent}(iii). Using Theorem~\ref{global IGDm}(ii) with taking into account that $e_\lambda h$ satisfies the PLK property at some accumulation point of $\set{x^k}$ yields the convergence $x^k\rightarrow \bar x$ as $k\rightarrow\infty.$ The convergence rates in (iii) follow directly from Theorem~\ref{global IGDm}(iii) and representation \eqref{prop 2.8 i}.
\end{proof}

The next result addresses the local convergence of Algorithm~\ref{IPP algorithm} to a local minimizer of problem \eqref{PPM problem}.

\begin{Theorem}\label{lo IPPm}
Let $h:\R^n\rightarrow\overline{\R}$ be a proper, l.s.c., prox-bounded, and semialgebraic function, which is prox-regular at some local minimizer $\bar x$ for $\bar v=0$. Then there is $\bar \lambda>0$ such that for any $\lambda\in (0,\bar \lambda]$, we can select $r_\lambda>0,\;\xi>0,\;T>0$ so that for $X=\B(\bar x,r_\lambda)$, $x^0=x^1\in \mathbb{B}(\bar x,\xi)$, $\bar \beta\in [0,1-\nu)$, and $\tau\in (0,T]$ in  Algorithm~{\rm\ref{IPP algorithm}}, the following local convergence properties hold: 
\begin{enumerate}[\bf(i)]
\item $\set{x^k},\set{x^k_{\rm in}},\set{x^k_{\rm ex}}$ converge to some local minimizer $\tilde x$ of $h$ with $h(\tilde x)=h(\bar x).$ 

\item If $e_\lambda h$ satisfies the exponential PLK condition at $\bar x$ with the desingularizing function $\psi(t)=Mt^{1-q}$ as $M>0$ and $q\in[0,1),$ then the local convergence rates are: 
\begin{itemize}
\item For $q\in[0,1/2]$, the sequence $\set{x^k}, \set{x^k-\prox_{\lambda h}(x^k)}$ converge linearly to $\tilde x$, $0$ respectively.

\item For $q\in(1/2,1)$, we have  the asymptotic estimates
$$
\norm{x^k-\tilde x}=\mathcal{O}\brac{k^{-\frac{1-q}{2q-1}}}\;\mbox{ and }\;\norm{x^k-\prox_{\lambda h}(x^k)}=\mathcal{O}\brac{k^{-\frac{1-q}{2q-1}}}\;\mbox{ as }\;k\to\infty
$$
\end{itemize}
\end{enumerate}
\end{Theorem}
\begin{proof}
Since  $\bar x$ is a local minimizer of $h,$ it follows from Proposition~\ref{localMoreau} that there exists $\bar \lambda>0$  such that $\bar x$ is a local minimizer of $e_\lambda h$ for all $\lambda\in (0,\bar \lambda].$ By choosing a smaller $\bar \lambda$ if necessary, Proposition~\ref{PPM problem transformation} tells us that for any $\lambda\in (0,\bar\lambda],$ we can find $r_\lambda>0$ for which $e_\lambda h$ is $\C^1$-smooth on $\B(\bar x,r_\lambda)$ and Algorithm~\ref{IPP algorithm} applied to $h$ reduces to Algorithm~\ref{IGDm} applied to $e_\lambda h$ with $X=\B(\bar x,r_\lambda).$ Fix such $\lambda$ and the corresponding $r_\lambda.$ By choosing a smaller $r_\lambda$, we get that $\bar x$ is a local minimizer of both $f$ and $h$ with respect to $\B(\bar x,r_\lambda).$ Since $h$ is a l.s.c.\ and 
semialgebraic function, its Moreau envelope $f=e_\lambda h$ is also semialgebraic; see, e.g., \cite[Page~1898]{Ioffe2009}. This verifies the basic PLK condition for $e_\lambda h$ at $\bar x$. 

Thus it follows from Theorem~\ref{local con IGDm} that there exist $\xi \in(0,r_\lambda]$ and $T>0$ such that for any $x^0=x^1\in\B(\bar x,\xi)$ and $\tau\in (0,T],$ we have the convergence properties of Theorem~\ref{local con IGDm}. This means that the sequences $\set{x^k},\set{x^k_\rin}$, and $\set{x^k_\rex}$ stay in $\mathbb{B}(\bar x,r_\lambda)$ and converge to a local minimizer $\tilde x\in \mathbb{B}(\bar x,r_\lambda)$ of $f$ with $f(\tilde x)=f(\bar x).$ By the classical Fermat rule, we have that $\nabla f(\tilde x)=0$, which implies by \eqref{GradEnvelope} that $\tilde x=\prox_{\lambda h}(\tilde x).$ Combining the latter with $f(\tilde x)=f(\bar x)$ yields
\begin{align*}
h(\bar x)&\ge f(\bar x)=f(\tilde x)=\inf_{x\in\R^n}\set{h(x)+\frac{1}{2\lambda}\norm{x-\tilde x}^2}\\
&=h(\prox_{\lambda h}(\tilde x))+\frac{1}{2\lambda}\norm{\prox_{\lambda h}(\tilde x)-\tilde x}^2=h(\tilde x).
\end{align*}
Since $\bar x$ minimizes $h$ relative to $\mathbb{B}(\bar x,r_\lambda)\ni \tilde x$, we get that $h(\tilde x)= h(\bar x),$ and hence $\tilde x$ is a local minimizer of $h$, which verifies (i). The convergence rates in (ii) follow from the convergence rates in Theorem~\ref{local con IGDm}(ii) with taking into account the gradient representation in \eqref{GradEnvelope}.
\end{proof}\vspace*{-0.2in}

\section{Numerical Experiments}\label{sec: numerical}

In this section, we present and discuss numerical experiments to confirm the efficiency of IGDm in Algorithm~\ref{IGDm} to solve problem \eqref{main problem}, where $f:\R^n\rightarrow\R$ is a function of class $\C^{1,1}_L$. The momentum settings considered below are as follows:\vspace*{0.05in}

{\bf(i)} Polyak's heavy ball method \cite{Polyak1964,ZavrievKostyuk1993}: $\beta_k=\frac{k}{k+3}$ and $\gamma_k=0$ for all $k\in\N.$

{\bf(ii)} Nesterov's acceleration for convex functions \cite{Nesterov1983,ShiDuJordanSu2022}: $\beta_k=\gamma_k=\frac{k}{k+3}$ for all $k\in\N.$\vspace*{0.05in}

The numerical experiments aim to demonstrate that IGDm, in addition to its universality that covers many well-known and efficient first-order algorithms, also has valuable applications in problems of derivative-free optimization. In this context, since the exact gradient is unavailable, we utilize inexact gradients through approximation methods. A comprehensive convergence analysis of such methods under the global and local Lipschitz continuity of the gradient is presented in the recent paper \cite{KhanhMordukhovichTrana}. In this section, we consider the standard finite-difference methods with: \vspace*{0.05in}

$\bullet$ {\em Forward finite difference}: 
\begin{align}\label{forward}
\mathcal{G}(x,\delta):=\dfrac{1}{\delta}\sum_{i=1}^n\brac{f(x+\delta e_i)-f(x)}e_i\;\text{ for any }\;(x,\delta)\in\R^n\times (0,\infty).
\end{align}

$\bullet$ {\em Central finite difference}: 
\begin{align}\label{central}
\mathcal{G}(x,\delta):=\dfrac{1}{2\delta}\sum_{i=1}^n\brac{f(x+\delta e_i)-f(x-\delta e_i)}e_i\;\text{ for any }\;(x,\delta)\in\R^n\times (0,\infty).
\end{align}
The following error bounds are known for both approximation schemes given above \cite{KhanhMordukhovichTrana}.
\begin{Proposition}\label{for err bound}
The forward  \eqref{forward} and central \eqref{central} finite differences satisfy the estimate
\begin{align}\label{glo err bound for}
\norm{\mathcal{G}(x,\delta) -\nabla f(x)}\le \dfrac{L\sqrt{n}\delta}{2}\; \text{ for any }\;(x,\delta)\in\R^n\times(0,\infty).
\end{align}
\end{Proposition}
Combining finite-difference approximations and inexact gradient descent methods, with or without momentum, brings us to the resulting devices:
\begin{itemize}
\item DF-fordif: Derivative-free method with forward differences.
\item DF-cendif: Derivative-free method with central differences.
\item DF-fordif: Derivative-free method with Nesterov's momentum and forward differences.
\item DFn-cendif: Derivative-free method with Nesterov's momentum and central differences.
\item DFp-fordif: Derivative-free method with Polyak's momentum and forward differences.
\item DFp-cendif: Derivative-free method with Polyak's momentum and central differences.
\end{itemize}
For each of these devices,  we proceed as follows. At the $k^{\rm th}$ iteration of Algorithm~\ref{IGDm}, select a reduction factor $\theta\in (0,1)$ and some accuracy $\varepsilon_k>0,$  and then generate an approximation $g^k= \mathcal{G}(x^k,\theta^{i_k}\varepsilon_k)$ with the smallest integer $i_k$ satisfying the estimate
\begin{align}\label{condition deri}
\norm{g^k-\nabla f(x^k_\rex)}\le \nu \norm{g^k}.
\end{align}
The existence and an upper bound for $i_k$ are provided in \cite[Remark~3.1]{KhanhMordukhovichTran}. In our numerical experiments, we consider functions $f$ in the following forms:
\begin{enumerate}
\item {\em Least-square regression}
\begin{align}\label{least square}
f(x):=\norm{Ax-b}^2,
\end{align}
where $A$ is an $n\times n$ matrix and $b$ is a vector in $\R^n.$ This problem is used for benchmarking derivative-free optimization methods in \cite{RiosSahinidis2013}.  It can be seen that $\nabla^2 f(x)=2A^*A\text{ for all }x\in\R^n,$ where $A^*$ is the transpose of $A$. Therefore, $\nabla f$ is Lipschitz continuous with the constant $L=2\norm{A^*A},$ where $\norm{M}:=\sigma_{\max}(M)$ is the largest singular value of the matrix $M\in\R^{n\times n}$.
\item {\em Nonconvex image restoration}
\begin{align}\label{nonconvex}
f(x):=\sum_{i=1}^n \log(1+(Ax-b)^2_i),
\end{align}
where $A$ is an $n\times n$ matrix and $b$ is a vector in $\R^n.$ This problem is considered in \cite[Section~5.5]{ThemelisStellaPatrinos2017} with a nonsmooth term added to the objective function. It is clear that $\nabla f$ is Lipschitz continuous with the constant $L=2\norm{A^*A}_\infty$, where
\begin{align*}
\norm{M}_\infty:=\max_{1\le i\le n}\sum_{j=1}^n|m_{ij}|\;\mbox{ for }\; M=[m_{ij}]_{i,j=\overline{1,n}}.
\end{align*}
\end{enumerate}
The aforementioned methods are tested on randomly generated datasets of different sizes. To be more specific, an $n\times n$ matrix $A$ and a vector $b\in\R^n$ are generated randomly with i.i.d. (independent and identically distributed) standard Gaussian entries. Behaviors of the methods are investigated in both the noise-free setting when the function value $f(x)$ is accessible at every $x\in\R^n$, and the noiseless setting when only $\phi(x):=f(x)+\xi(x)$ is accessible but not $f(x)$, where $\xi(x)$ is a uniformly distributed random variable given by $\xi(x)\sim U(-10^{-4},10^{-4}).$ We also test the methods for different dimensions $n\in\set{50,100}$. In total, the methods are tested for 8 different problems. All the methods  under consideration are executed until they reach the maximum number of function evaluations of $200n$. The numerical results are presented in Figure~\ref{fig:C1L 50 result} below, where $L$ stands for the least-square regression in \eqref{least square} and $N$ signifies the nonconvex image restoration in \eqref{nonconvex}.

The obtained results show that the momentum effect does improve the performance of all the derivative-free methods above. In the noiseless setting, DFn-fordif can be considered the best while typically enhancing the speed of the standard DF-fordif by a factor of $2$. In the noisy setting, the central difference is more stable than the forward one and show that  DFn-cendif is the best with about 2-times speed enhancement over the standard DF-cendif. Another observation from the presented results is that in the noisy setting, due to a simpler structure, solving least-square problems is easier than solving nonconvex image restoration problems while giving us better results for all the algorithms.\vspace*{-0.15in}

\section{Concluding Remarks and Future Research}\label{conc}

This paper designs and provides a comprehensive local and global convergence analysis of the novel inexact gradient decent method with momentum. Although the basic problem under consideration is given in the framework of smooth unconstrained optimization, it is shown that the developed algorithm can be used for some classes of nonsmooth constrained optimization problems written in the extended-real-valued format. Several specifications and applications of the main algorithm are also investigated with conducting numerical experiments for problems of derivative-free optimization.

Our future research on these topics will go in several directions. In particular, we plan to cover optimization problems with explicit constraints given by smooth and nonsmooth functions and generally nonconvex sets and also extend the algorithms to problems of stochastic optimization. A major attention will be paid to applications of the designed algorithms for practical models arising in data science, artificial intelligence, machine learning, statistics, etc.

\begin{figure}[H]
\centering
\includegraphics[width=.45\textwidth]{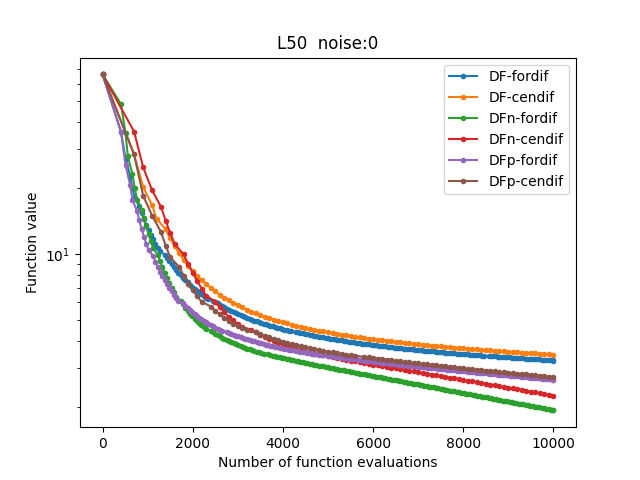}\quad
\includegraphics[width=.45\textwidth]{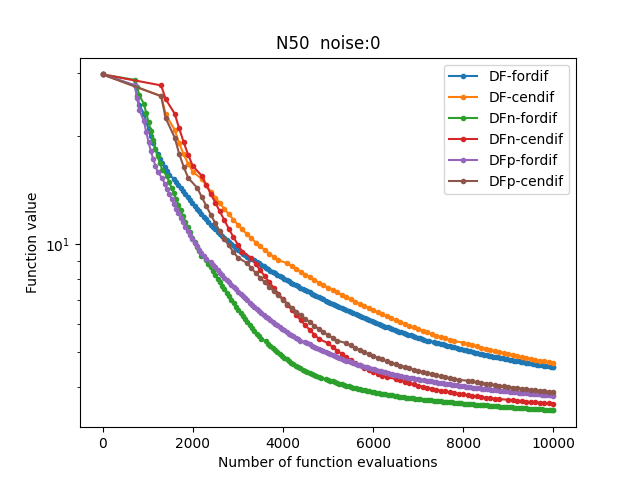}\quad

\includegraphics[width=.45\textwidth]{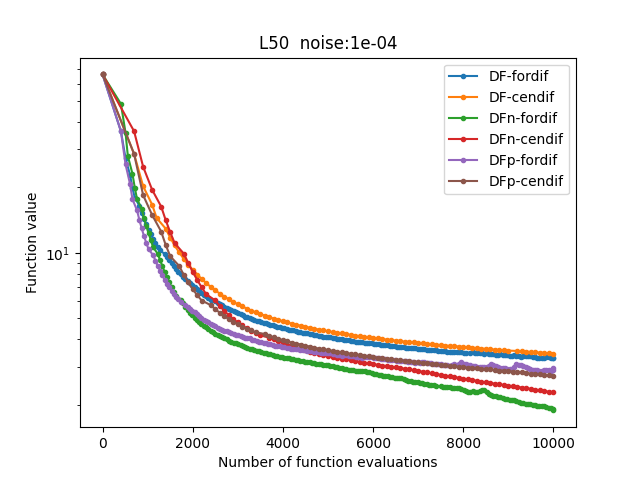}\quad
\includegraphics[width=.45\textwidth]{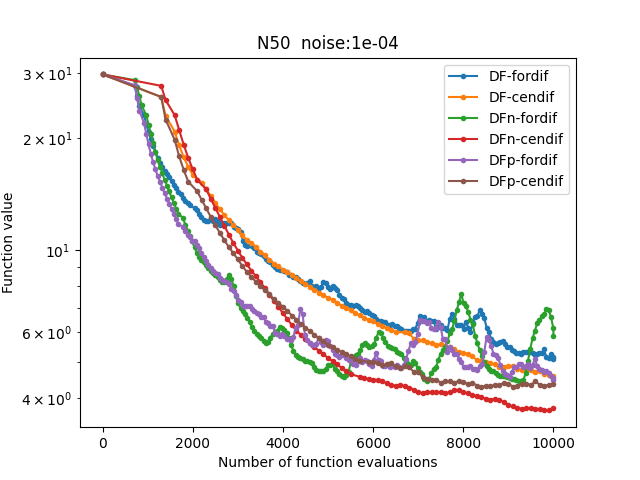}\quad

\includegraphics[width=.45\textwidth]{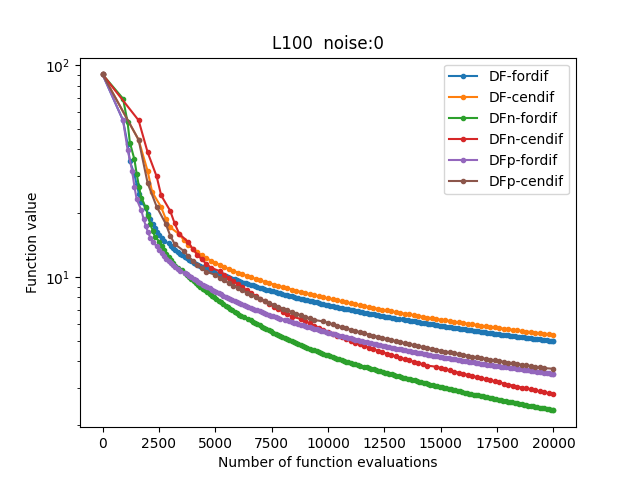}\quad
\includegraphics[width=.45\textwidth]{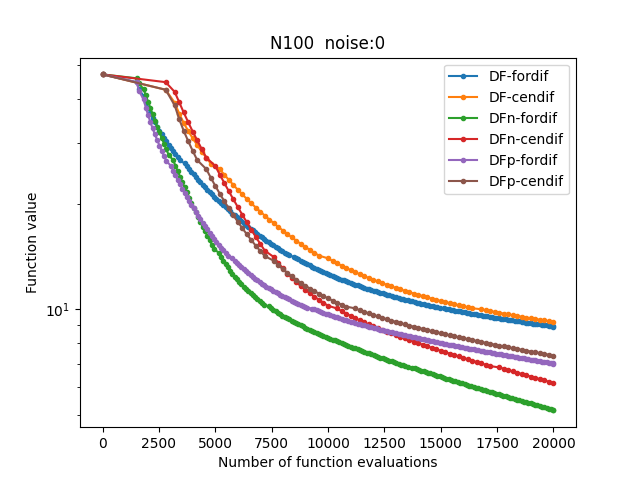}\quad

\includegraphics[width=.45\textwidth]{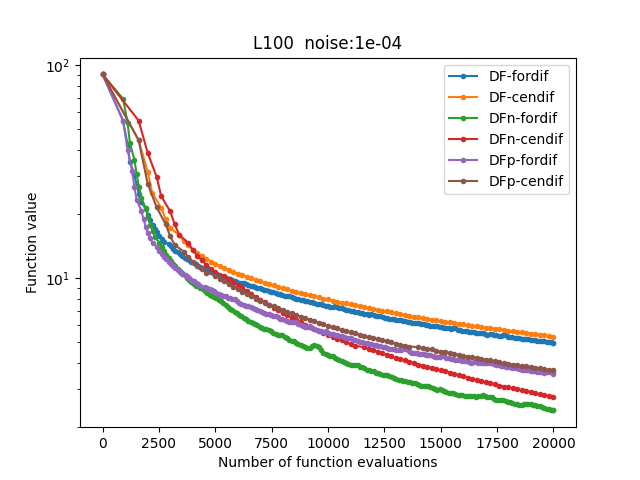}\quad
\includegraphics[width=.45\textwidth]{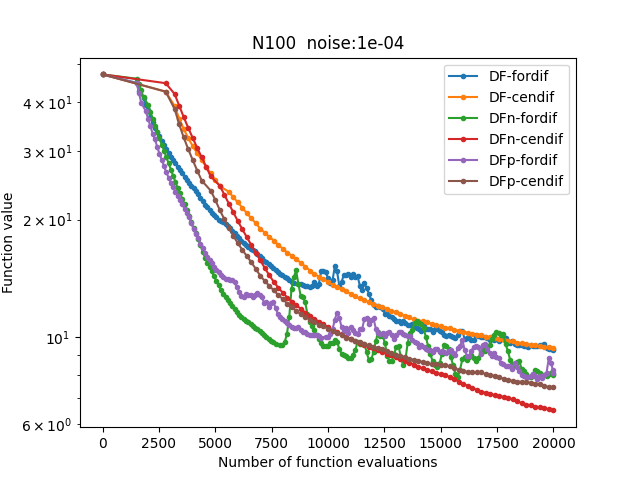}\quad
 \caption{Derivative-free least-square convex and nonconvex problems}\label{fig:C1L 50 result}
\end{figure}
\section*{}
\bibliographystyle{plain} 
\bibliography{bibtex}

\end{document}